\documentclass[12 pt]{amsart}
\usepackage[T1]{fontenc}
\usepackage{amsmath}
\usepackage{amsthm}
\usepackage{amssymb}
\usepackage{amsfonts}
\usepackage{outlines}
\usepackage{dsfont}
\usepackage{stackengine}
\usepackage{comment}
\usepackage{cancel}
\usepackage{graphicx}
\usepackage{enumerate}
\usepackage{mathrsfs}
\usepackage{tikz-cd}
\usepackage{hyperref}
\usepackage[numbers]{natbib}
\usepackage{enumitem}
\usepackage{lipsum}
\usepackage[left = 3cm, right = 3cm, top = 4 cm, bottom = 4 cm]{geometry}

\newtheorem{thm}{Theorem} [section]

\newtheorem{lemma}[thm]{Lemma}
\newtheorem{cor}[thm]{Corollary}
\newtheorem{remark}[thm]{Remark}
\newtheorem{prop}[thm]{Proposition}

\title[Uniqueness of tangent flows in free boundary flow]{Uniqueness of tangent flows in free boundary flow}
\author{Sourav Ghosh}
\date{}

\address{Department of Mathematics, University of Notre Dame, Notre Dame, IN 46556}
\email{sourav.ghosh.1@icloud.com}

\begin{document}

\maketitle

\begin{abstract}
We prove uniqueness of tangent flows for mean curvature flow with free boundary at singularities modeled on half-cylindrical self-shrinkers. More precisely, if a tangent flow is of the form $\mathbb{S}^{n-k} \times \mathbb{R}^k_+$ or $\mathbb{S}^{n-k}_+ \times \mathbb{R}^k $, then it is unique. This provides the free-boundary analogue of the uniqueness theory for cylindrical singularities. 
\end{abstract}


\section{Introduction}\label{section 1}

Mean curvature flow is one of the fundamental geometric evolution equations and has been studied extensively over the past several decades. Given a smooth hypersurface $M_0\subset\mathbb{R}^{n+1}$, the flow evolves the hypersurface in the direction of its mean curvature vector. A natural extension of the classical problem is the free boundary setting, where the evolving hypersurfaces are contained in a fixed domain $\Omega\subset\mathbb{R}^{n+1}$, satisfy the mean curvature flow equation in the interior of $\Omega$, and meet the boundary $\partial\Omega$ orthogonally. The study of free boundary mean curvature flow was initiated by Stahl [\citealp{stahl1996convergence}, \citealp{stahl1996regularity}] and Buckland \cite{buckland2005mean}, and has since developed into a rich area of geometric analysis.

As in the classical setting, singularities generally form in finite time, and their occurrence necessitates the use of weak formulations of the flow. Over the years, several such notions have been introduced, including the free boundary level set flow \cite{giga1993neumann} and the free boundary Brakke flow \cite{edelen2020free}. These frameworks make it possible to study the evolution past singularities and have led to significant advances in the understanding of free boundary mean curvature flow. In particular, recent work of Edelen-Ivaki-Zhu-Haslhofer \cite{edelen2022mean} established a detailed regularity and structure theory for mean-convex free boundary Brakke flows, providing a foundation for the analysis of singularities. More recently, Bao-Haslhofer \cite{bao2026free} established a well-posedness theory for free boundary mean curvature flow through cylindrical and half-cylindrical singularities. They proved that singularities modeled on $\mathbb{S}^{n-k} \times \mathbb{R}^k, \mathbb{S}^{n-k} \times \mathbb{R}^k_+, \text{or} \; \mathbb{S}^{n-k}_+ \times \mathbb{R}^k$ have a mean-convex neighborhood. Moreover, they showed that the free boundary level set flow is nonfattening provided all singularities have a mean-convex neighborhood. As a consequence, they established that free boundary mean curvature flow is well-posed through singularities of cylindrical and half-cylindrical type. Their work provides a natural framework for studying the finer structure of these singularities, including the uniqueness of tangent flows.

However, these recent advances represent only the beginning of a comprehensive
singularity theory for free boundary mean curvature flow. A fundamental tool in
the study of singularities is the notion of a tangent flow, obtained as the
limit of suitable parabolic rescalings about a singular point. Tangent flows
capture the asymptotic geometry of the evolving hypersurface near a singularity. Although the existence of tangent flows follows from compactness, their uniqueness is a substantially more delicate question.

In the boundaryless setting, fundamental work of Schulze \cite{schulze2014uniqueness}, Colding-Minicozzi \cite{colding2015uniqueness}, and Chodosh-Schulze \cite{chodosh2021uniqueness} established uniqueness results for compact, cylindrical, and asymptotically conical self-shrinkers, respectively. By contrast, analogous uniqueness results in the free boundary setting are largely unavailable. In this paper, we establish the free boundary analogue of the uniqueness theorem for cylindrical tangent flows.

At a boundary singular point $x_0\in \partial\Omega$, the geometry of the
ambient domain becomes flat under parabolic blow-up. More precisely, after a
suitable rotation, the rescaled domains $\lambda(\Omega-x_0)$ converge to the
half-space
\[
\mathbb{R}^{n+1}_+ := \mathbb{R}^{n+1}\cap\{x_{n+1}\geq 0\}
\]
as $\lambda\to\infty$. Thus, the natural models for boundary singularities are
self-shrinkers in the half-space $\mathbb{R}^{n+1}_+$ satisfying the free
boundary condition. Boundary neckpinches provide natural examples of
singularities whose tangent flows are half-cylindrical self-shrinkers. There
are two types of half-cylindrical self-shrinkers, distinguished by whether the
cylindrical axis is orthogonal or tangent to the boundary hyperplane. Namely,
the two models are
\[
\mathbb{S}^{n-k}_{\sqrt{2(n-k)}}\times \mathbb{R}^{k}_{+} \qquad \text{and} \qquad  \mathbb{S}^{n-k}_{+,\sqrt{2(n-k)}}\times \mathbb{R}^{k},
\]
where $\mathbb{S}^{n-k}_{\sqrt{2(n-k)}}$ denotes the round shrinking sphere of radius
$\sqrt{2(n-k)}$ and $\mathbb{S}^{n-k}_{+,\sqrt{2(n-k)}}$ denotes the hemisphere of radius $\sqrt{2(n-k)}$ meeting the boundary hyperplane orthogonally. The first model has its cylindrical axis orthogonal to the boundary, while the second has its cylindrical axis tangent to the boundary.

Our main result is the following.

\begin{thm} \label{thm 1.1}
Let $\mathcal{M} = (M_t)_{t \in (-1,0)}$ be a mean curvature flow with free boundary, and let
$(x_0,t_0)$ be a boundary singular point. Suppose that one tangent flow at
$(x_0,t_0)$ is given by one of the half-cylindrical self-shrinkers, that is, one of the two models
\[
\mathbb{S}^{n-k}_{\sqrt{2(n-k)}}\times \mathbb{R}^{k}_{+}
\quad\text{or}\quad
\mathbb{S}^{n-k}_{+,\sqrt{2(n-k)}}\times \mathbb{R}^{k}.
\]
Then the tangent flow at $(x_0,t_0)$ is unique.
\end{thm}

Among the three types of cylindrical singularities, the model $\mathbb{S}^{n-k}\times\mathbb{R}^k$ arises at interior singular points. The uniqueness of tangent flows in this case follows from the result of Colding-Minicozzi \cite{colding2015uniqueness}, as mentioned above; see also Zhu \cite{zhu2020ojasiewicz}, as well as \cite{ghosh2025cylindrical}, and Bamler-Lai \cite{bamler2025pde}. Together with Theorem~\ref{thm 1.1}, this settles the uniqueness of tangent flows for all three cylindrical models.

\subsection*{Acknowledgments}
I would like to express my sincere gratitude to Professor G\'abor Sz\'ekelyhidi for introducing me to this problem. I am also grateful to Professor Nicholas Edelen for his many insightful discussions and valuable suggestions, which have significantly contributed to the development of this work.

\section{Preliminaries}\label{section 2} 

In this section, we briefly discuss the basic definitions in mean curvature flow with free boundary. Throughout the
paper, $\Omega\subset\mathbb{R}^{n+1}$ denotes a smooth domain, and we consider smooth, properly immersed hypersurfaces $\Sigma^n\subset\overline{\Omega}$ satisfying the free boundary condition
\[
\Sigma\perp\partial\Omega \qquad\text{along }\partial\Sigma.
\]

We denote by $x$ the position vector in $\mathbb{R}^{n+1}$. For a vector $V$, we write $V^T$ for its projection onto the tangent bundle $T\Sigma$, and $V^\perp=\Pi(V)$ for its projection onto the normal bundle $N\Sigma$.

The second fundamental form is the symmetric $2$-tensor with values in the normal bundle defined by $A(Y,Z)=(\nabla_YZ)^\perp,$ and the mean curvature vector is given by $H=\operatorname{tr}_{\Sigma}A.$ We define the shrinker mean curvature by $\phi=H+\frac12x^\perp.$ 

For $\lambda>0$, we define the parabolic rescaling
\begin{align*}
\mathcal{D}_\lambda:\mathbb{R}^{n+1}\times\mathbb{R} &\longrightarrow \mathbb{R}^{n+1}\times\mathbb{R},\\
(x,t) &\longmapsto (\lambda x,\lambda^2t).
\end{align*}
If $\mathcal{M}$ denotes the space-time track of a mean curvature flow, then $\mathcal{D}_\lambda\mathcal{M}$ is again a mean curvature flow.

Suppose that a free boundary mean curvature flow develops a singularity at the space-time point $(x_0,T)$. For any sequence $\lambda_i\to\infty$, the rescaled flows
\[
\mathcal{D}_{\lambda_i}\bigl(\mathcal{M}-(x_0,T)\bigr)
\]
admit a subsequence converging as Brakke flows to a limit $\{\nu_t\}_{t\in(-\infty,0)}$. If $x_0\in\Omega$, then the limit is a Brakke flow in $\mathbb{R}^{n+1}$. If $x_0\in\partial\Omega$, then, after a suitable
rotation, the limit is a free boundary Brakke flow in the half-space
\[
\mathbb{R}^{n+1}_+ = \mathbb{R}^{n+1}\cap\{x_{n+1}\ge0\}.
\]
Any such limit is called a tangent flow at $(x_0,T)$. In general, the tangent flow may depend on the sequence $\lambda_i$. Moreover, every tangent flow is a self-shrinking Brakke flow \cite{edelen2020free}.

Let $(M_t)_{0\le t<T}$ be a free boundary mean curvature flow and let $x_0\in\overline{\Omega}$. The associated rescaled flow centered at $(x_0,T)$ is
\[
M_\tau = \frac{1}{\sqrt{T-t}}(M_t-x_0), \qquad \tau=-\log(T-t).
\]
A straightforward computation shows that the rescaled flow evolves with normal velocity $H+\frac12x^\perp.$ Thus, tangent flows may equivalently be obtained as limits of sequences $M_{\tau_i}$ with $\tau_i\to\infty$.

\subsection{Fermi Coordinates}\label{section 2.1} 

We may assume, without loss of generality, that the singularity occurs at the space-time origin \((0,0) \in \mathbb{R}^{n+1} \times \mathbb{R}\), where \(0\in S=\partial\Omega\). We work in Fermi coordinates centered at the boundary point \(0\), chosen so that
\[
S=\{x_{n+1}=0\}, \qquad \Omega=\{x_{n+1}\geq 0\}.
\]
Let \(\nu\) denote the outward unit normal to \(S\). Thus, the positive \(x_{n+1}\)-direction points into \(\Omega\). Let
\[
X:B_{r_0}^{n}(0)\longrightarrow S
\]
be a local parametrization of \(S\) satisfying \(X(0)=0\). For \(r_0>0\) sufficiently small, depending only on the \(C^{3,\alpha}\)-geometry of \(S\), define the Fermi coordinate map
\[
\begin{aligned}
\Phi:\quad B_{r_0}^{n}(0)\times(-r_0,r_0)&\longrightarrow\mathbb{R}^{n+1},\\
(x',x_{n+1})&\longmapsto X(x')-x_{n+1}\nu(x').
\end{aligned}
\]
For \(r_0\) sufficiently small, \(\Phi\) is a diffeomorphism onto a tubular neighborhood of \(0\). After identifying this neighborhood with its coordinate image, we further decrease \(r_0\), if necessary, so that
\[
B_{r_0}^{n+1}(0) \subset B_{r_0}^{n}(0)\times(-r_0,r_0).
\]
We use these Fermi coordinates on \(B_{r_0}^{n+1}(0)\) and denote by
\[
g=\Phi^*\delta
\]
the pullback of the Euclidean metric. We choose the tangential coordinates so that
\[
g(0)=\delta.
\]
The metric coefficients satisfy
\[
g_{ij}(x',x_{n+1}) = \sigma_{ij}(x') + 2x_{n+1}A^S_{ij}(x') + x_{n+1}^2 \left\langle \partial_i\nu,\partial_j\nu \right\rangle,
\]
for $i,j =1,..., n.$ In addition, the Fermi coordinates satisfy  
$$g_{i,n+1} =0,\qquad g_{n+1,n+1} =1.$$
throughout the tubular neighborhood. Here \(A^S_{ij}\) denotes the components of the second fundamental form of \(S\) with respect to the outward unit normal \(\nu\).

Throughout, all tensor norms are taken with respect to the Euclidean metric \(\delta\). Since \(S\) is \(C^{2,\alpha}\), after possibly decreasing \(r_0\), there exists a constant \(C>0\), depending only on the \(C^{2,\alpha}\)-geometry of \(S\), such that in \(B_{r_0}(0)\),
\[
|g-\delta|\leq C|x|, \qquad |\partial g|\leq C, \qquad |\partial^2 g|\leq C, \qquad |\Gamma_g|\leq C.
\]
After decreasing \(r_0\) further if necessary, we may also assume that
\[
\frac12\delta\leq g\leq 2\delta, \qquad |g^{-1}-\delta|\leq C|x|.
\]
Here the inequalities between metrics are understood in the sense of quadratic forms. In particular, the Fermi metric is uniformly equivalent to the Euclidean metric in \(B_{r_0}(0)\).

We now introduce the parabolic rescaling. For \(\tau\in\mathbb{R}\), define
\[
M_\tau=e^{\frac{\tau}{2}}M_{-e^{-\tau}}.
\]
Here the rescaling is understood in the above Fermi coordinate system around the origin. More precisely, we identify \(M_t\) locally with its coordinate representation in the Fermi chart and perform the dilation $x\mapsto e^{\frac{\tau}{2}}x.$ Let
\[
\widetilde{X}(p,\tau) = e^{\frac{\tau}{2}}X(p,-e^{-\tau})
\]
be the parabolically rescaled flow. Since the coordinate representation of \(M_t\) is contained in \(B_{r_0}(0)\), the rescaled hypersurface \(M_\tau\) is contained in \(B_{e^{\frac{\tau}{2}}r_0}(0)\). The corresponding rescaled ambient metric is given by
\[
g_\tau(x) = g(e^{-\frac{\tau}{2}}x).
\]
The map \(\widetilde{X}(\cdot,\tau)\) is then viewed as an immersion into the ambient manifold equipped with the metric \(g_\tau\).

The rescaled flow satisfies the mean curvature flow equation with respect to the metric \(g_\tau\), together with the additional drift term induced by the dilation,
\[
(\partial_\tau \widetilde{X})^\perp = \left(H_{g_\tau} + \frac12 \langle \widetilde{X}, \nu_{g_\tau} \rangle_{g_\tau} \right) \nu_{g_\tau}.
\]
Here \(H_{g_\tau}\) denotes the scalar mean curvature of \(\widetilde{X}\) with respect to \(g_\tau\), \(\nu_{g_\tau}\) is the \(g_\tau\)-unit normal, and \((\cdot)^\perp\) denotes the normal projection with respect to \(g_\tau\).

The purpose of introducing Fermi coordinates and then passing to the parabolically rescaled flow is to flatten the free boundary. Suppose that, along a sequence $\tau_i\to\infty$, the rescaled flows $M_{\tau_i}$ converge locally smoothly to the half-cylinder $\mathscr C_+$. In the above coordinates, both $M_{\tau_i}$ and $\mathscr C_+$ have boundary contained in the fixed hyperplane $S=\{x_{n+1}=0\},$ while the rescaled metrics $g_{\tau_i}$ converge locally smoothly to the Euclidean metric $\delta$.

Consequently, for every $R>0$, and for all sufficiently large $i$, $M_{\tau_i}\cap B_R(0)$ can be represented as a normal graph over $\mathscr C_+\cap B_R(0)$. The fact that the boundary is represented by the same fixed hyperplane $S$ allows this graphical representation to be carried out up to the boundary.

We now record the estimates for the rescaled metric \(g_\tau\). For the zeroth order term, using the estimate for \(g-\delta\), we obtain
\[
|g_\tau(x)-\delta| = |g(e^{-\frac{\tau}{2}}x)-\delta| \leq Ce^{-\frac{\tau}{2}}|x|.
\]
Therefore, for every \(R>0\) satisfying $R\leq e^{\frac{\tau}{2}}r_0,$ on \(B_R(0)\),
\[
|g_\tau-\delta| \leq CRe^{-\frac{\tau}{2}}.
\]
For the first derivatives, we have
\[
\partial_i(g_\tau)_{\alpha\beta}(x) = e^{-\frac{\tau}{2}}(\partial_i g_{\alpha\beta})(e^{-\frac{\tau}{2}}x),
\]
and hence, on \(B_R(0)\),
\[
|\partial g_\tau| \leq Ce^{-\frac{\tau}{2}}.
\]
For the second derivatives, we have
\[
\partial_i\partial_j(g_\tau)_{\alpha\beta}(x) = e^{-\tau}(\partial_i\partial_j g_{\alpha\beta})(e^{-\frac{\tau}{2}}x),
\]
which gives, on \(B_R(0)\),
\[
|\partial^2 g_\tau| \leq Ce^{-\tau}.
\]
Here the constant \(C>0\) depends only on the \(C^{2,\alpha}\)-geometry of \(S\) and is independent of \(R\) and \(\tau\).

Since \(g_\tau(x)=g(e^{-\frac{\tau}{2}}x)\), the corresponding estimates for \(g\) imply
\[
\frac12\delta\leq g_\tau\leq2\delta, \qquad |g_\tau^{-1}-\delta| \leq Ce^{-\frac{\tau}{2}}|x|
\]
on \(B_R(0)\). In particular,
\[
|g_\tau^{-1}|\leq C.
\]
Originally, the free boundary condition is
\[
\langle \nu_{M_t},\nu_{\partial\Omega}\rangle=0 \qquad\text{on }M_t\cap\partial\Omega .
\]
After passing to the above Fermi coordinates, the boundary is flattened as
\[
S=\partial\Omega=\{x_{n+1}=0\},
\]
and the free boundary condition becomes
\[
\langle \nu_{M_t},\nu_{S}\rangle_g=0 \qquad\text{on }M_t\cap S,
\]
where \(\nu_S\) and \(\nu_{M_t}\) denote the unit normals with respect to the metric \(g\).

After the parabolic rescaling, the boundary remains
\[
S=\{x_{n+1}=0\},
\]
and the free boundary condition for \(M_\tau\) is
\[
\langle \nu_{M_\tau},\nu_S\rangle_{g_\tau}=0 \qquad\text{on }M_\tau\cap S.
\]

\subsection{Graphical Representation} \label{section 2.2} 

We first recall the graphical formulation for the rescaled mean curvature flow near a cylindrical self-shrinker \(\mathscr{C}\) [\citealp{colding2015uniqueness}, Lemma ~A.44, and also \citealp{colding2019dynamics}]. Consider the linearized operator
\[
Lu = \Delta u -\frac12 x\cdot\nabla u + |A|^2u + \frac12 u .
\]
On the cylinder \(\mathscr{C}\), we have $|A|^2=\frac12,$ and hence the operator reduces to
\[
Lu = \Delta u -\frac12 x\cdot\nabla u + u .
\]
There exists $c >0$ such that for \(u\in C^{2,\alpha}(\mathscr{C})\) with $\|u\|_{C^{2,\alpha}(\mathscr{C})}\le c,$ the rescaled mean curvature flow written as a normal graph of the function $u$ over
\(\mathscr{C}\) takes the form
\[
\partial_\tau u=Lu+Q(u),
\]
where \(L\) is the linearization of the rescaled mean curvature operator at \(u=0\), and $Q(0)=0, DQ(0)=0.$ In particular, \(Q\) vanishes to second order at the origin. More precisely, the nonlinear term $Q$ satisfies the estimate
\[
|Q(u)| \le C\Big(|\nabla u|^4 + |\nabla u|^2 |\nabla^2 u| + |\nabla u|^2 + u^2 + |u|\,|\nabla^2 u|\Big),
\]
for all $u$ with $\|u\|_{C^{2,\alpha}(\mathscr{C})} \le c$, where $C>0$ depends only on $\mathscr{C}$.

We now return to the rescaled free boundary flow considered above. Write points of $\mathscr C = S^{n-k}_{\rho}\times\mathbb R^k$ as $p= (y,z)$, where $y\in\mathbb{R}^{n-k+1}, z\in\mathbb{R}^k.$ Suppose that $M_\tau$ is written as a normal graph over $\mathscr{C}$ with graph function $u$. Thus
\[
F(p,\tau) = p + u(p,\tau)\frac{y}{\sqrt{2(n-k)}}.
\]
We derive the evolution equation satisfied by $u$. We use the same parametrization and notation as in the corresponding Euclidean graph equation \cite{sun2022generic}. As in the Euclidean case,
\[
F_\alpha = \left(1+\frac{u}{\sqrt{2(n-k)}}\right)y_\alpha +\frac{\partial_\alpha u}{\sqrt{2(n-k)}}y, \qquad F_i = \partial_i+\frac{\partial_i u}{\sqrt{2(n-k)}}y.
\]
In particular,
\[
|F_a|\leq C(1 + |u| + |\nabla u|).
\]
Moreover,
\[
\begin{aligned}
F_{\alpha\beta}
&=
-\left(1+\frac{u}{\sqrt{2(n-k)}}\right) \frac{\delta_{\alpha\beta}}{2(n-k)}y +\frac{\partial_\beta u}{\sqrt{2(n-k)}}y_\alpha +\frac{\partial_\alpha u}{\sqrt{2(n-k)}}y_\beta +\frac{\partial_{\alpha\beta}u}{\sqrt{2(n-k)}}y,\\
F_{i\alpha}
&=
\frac{\partial_{i\alpha}u}{\sqrt{2(n-k)}}y +\frac{\partial_i u}{\sqrt{2(n-k)}}y_\alpha, \\
F_{ij}
&=
\frac{\partial_{ij}u}{\sqrt{2(n-k)}}y,
\end{aligned}
\]
and consequently
\[
|F_{ab}| \leq C\left(1 + |\nabla u|+|\nabla^2u|\right).
\]
Let 
\[
\widetilde g^E_{ab}=\delta(F_a,F_b)
\]
denote the induced Euclidean metric on the graph, while
\[
\widetilde g^\tau_{ab} = g_\tau(F_a,F_b)
\]
denotes the induced metric with respect to $g_\tau$. We have
\[
\widetilde g^\tau_{ab} - \widetilde g^E_{ab} = (g_\tau-\delta)(F_a,F_b),
\]
and hence
\[
\left|\widetilde g^\tau-\widetilde g^E \right| \leq CRe^{-\frac{\tau}{2}}(1 + |u| + |\nabla u|)^2.
\]
Since the Euclidean graph metric is uniformly elliptic when $|u|$ and $|\nabla u|$ are sufficiently small, it follows that
\[
\left|(\widetilde g^\tau)^{-1} - (\widetilde g^E)^{-1} \right|\leq CRe^{-\frac{\tau}{2}}(1 + |u| + |\nabla u|)^2.
\]
Let $\nu_E$ denote the Euclidean unit normal to the graph,
\[
\nu_E = \frac{\frac{y}{\sqrt{2(n-k)}} -\left(1+\frac{u}{\sqrt{2(n-k)}}\right)^{-1} \partial_\alpha u\,y_\alpha -\partial_i u\,\partial_i}{\sqrt{1+ \left(1+\frac{u}{\sqrt{2(n-k)}}\right)^{-2} |\nabla_y u|^2
+ |\nabla_zu|^2}}.
\]
We obtain $g_\tau$-normal by setting
\[
N_\tau = g_\tau^{-1}(\nu_E^\flat),
\]
where
\[
\nu_E^\flat=\delta(\nu_E,\cdot).
\]
Indeed, for every tangent vector $F_a$,
\[
g_\tau(N_\tau,F_a) = \nu_E^\flat(F_a) = \delta(\nu_E,F_a) = 0.
\]
Thus, $N_\tau$ is $g_\tau$-normal to the graph. Since, on $B_R(0), |g_\tau^{-1}-\delta| \leq CRe^{-\frac{\tau}{2}},$ we have
\[
|N_\tau-\nu_E| \leq CRe^{-\frac{\tau}{2}}.
\]
Moreover,
\[
\left||N_\tau|_{g_\tau}-1 \right| \leq CRe^{-\frac{\tau}{2}}.
\]
Therefore, writing $\nu_\tau = \frac{N_\tau}{|N_\tau|_{g_\tau}},$ we obtain
\[
|\nu_\tau-\nu_E| \leq CRe^{-\frac{\tau}{2}}.
\]
Let $\nabla^\tau$ denote the Levi-Civita connection of $g_\tau$. Since the background Euclidean Christoffel symbols vanish in the pulled-back coordinates,
\[
\Gamma^\tau = \frac12g_\tau^{-1}*\nabla g_\tau,
\]
and hence
\[
|\Gamma^\tau| \leq Ce^{-\frac{\tau}{2}}.
\]
It follows that
\[
\nabla^\tau_{F_a}F_b = F_{ab} + \Gamma^\tau(F_a,F_b),
\]
and therefore
\[
\left|\nabla^\tau_{F_a}F_b-F_{ab}\right| \leq Ce^{-\frac{\tau}{2}}(1 + |u| + |\nabla u|)^2.
\]
We define the second fundamental form with respect to $g_\tau$ by
\[
\widetilde A^\tau_{ab} = g_\tau(\nabla^\tau_{F_a}F_b,\nu_\tau),
\]
while
\[
\widetilde A^E_{ab} = \delta(F_{ab},\nu_E)
\]
is the Euclidean second fundamental form. We have
\[
\widetilde A^\tau_{ab}-\widetilde A^E_{ab} = (g_\tau-\delta) (\nabla^\tau_{F_a}F_b,\nu_\tau) + \delta(\nabla^\tau_{F_a}F_b-F_{ab}, \nu_\tau) + \delta(F_{ab},\nu_\tau-\nu_E).
\]
For the first term, using
\[
|\nabla^\tau_{F_a}F_b| \leq C(1+ |u| + |\nabla u|+|\nabla^2u|) + Ce^{-\frac{\tau}{2}}(1+ |u|+ |\nabla u|)^2,
\]
we obtain
\[
\left|(g_\tau-\delta) (\nabla^\tau_{F_a}F_b,\nu_\tau)\right| \leq CRe^{-\frac{\tau}{2}} (1+ |u| + |\nabla u|+|\nabla^2u|) + CRe^{-\tau}(1+ |u| + |\nabla u|)^2.
\]
For the second term,
\[
\left|\delta(\nabla^\tau_{F_a}F_b-F_{ab}, \nu_\tau) \right| \leq Ce^{-\frac{\tau}{2}}(1+ |u| + |\nabla u|)^2.
\]
Finally, using the estimate for the normal,
\[
\left|\delta(F_{ab},\nu_\tau-\nu_E) \right|\leq |F_{ab}|\,|\nu_\tau-\nu_E|\leq CRe^{-\frac{\tau}{2}}(1+ |u| + |\nabla u|+|\nabla^2u|).
\]
Consequently,
\[
|\widetilde A^\tau-\widetilde A^E| \leq CRe^{-\frac{\tau}{2}} (1+|u| + |\nabla u|+|\nabla^2u|) + Ce^{-\frac{\tau}{2}}(1+|u| +|\nabla u|)^2 + CRe^{-\tau}(1+|u|+|\nabla u|)^2.
\]
With the convention
\[
H_\tau = (\widetilde g^\tau)^{ab}\widetilde A^\tau_{ab},
\]
and
\[
H_E = (\widetilde g^E)^{ab}\widetilde A^E_{ab},
\]
we have
\[
H_\tau-H_E = ( (\widetilde g^\tau)^{ab} -(\widetilde g^E)^{ab}) \widetilde A^E_{ab} + (\widetilde g^\tau)^{ab} (\widetilde A^\tau_{ab}-\widetilde A^E_{ab}).
\]
From the Euclidean graph formulas,
\[
|\widetilde A^E| \leq C(1+ |u| + |\nabla u|+|\nabla^2u|).
\]
Therefore
\[
\begin{aligned}
|H_\tau-H_E| \leq{}& CRe^{-\frac{\tau}{2}} (1+ |u| + |\nabla u|)^2 (1+ |u| + |\nabla u|+|\nabla^2u|) 
\\&
+ CRe^{-\frac{\tau}{2}} (1+|u|+|\nabla u|+|\nabla^2u|) + Ce^{-\frac{\tau}{2}}(1+|u| + |\nabla u|)^2
\\&+
CRe^{-\tau}(1+ |u| + |\nabla u|)^2.
\end{aligned}
\]
We next consider the drift term. Since $X = p + \frac{u}{\sqrt{2(n-k)}}y,$ the Euclidean calculation gives
\[
\delta(X,\nu_E) = \sqrt{2(n-k)}+u-z_i\partial_i u.
\]
On the other hand,
\[
g_\tau(X,\nu_\tau) - \delta(X,\nu_E) = (g_\tau-\delta)(X,\nu_\tau) + \delta(X,\nu_\tau-\nu_E).
\]
Consequently,
\[
\left|g_\tau(X,\nu_\tau) - (\sqrt{2(n-k)}+u-z_i\partial_i u)\right| \leq CRe^{-\frac{\tau}{2}}|X||\nu_\tau|
+ |X|\,|\nu_\tau-\nu_E|.
\]
Since $\nu_\tau$ is $g_\tau$-unit,
\[
1=g_\tau(\nu_\tau,\nu_\tau) =|\nu_\tau|_\delta^2 +(g_\tau-\delta)(\nu_\tau,\nu_\tau).
\]
Hence, 
\[
|\nu_\tau|_\delta\leq C.
\]
Therefore,
\[
\left|g_\tau(X,\nu_\tau) - (\sqrt{2(n-k)}+u-z_i\partial_i u) \right| \leq CRe^{-\frac{\tau}{2}}|X|.
\]
Since the graph is contained in the Fermi coordinate ball $B_R(0)$, we have
\[
\left|g_\tau(X,\nu_\tau) - (\sqrt{2(n-k)}+u-z_i\partial_i u) \right|\leq CR^2e^{-\frac{\tau}{2}}.
\]
Since
\[
\partial_\tau X = \frac{\partial_\tau u}{\sqrt{2(n-k)}}y = \partial_\tau u\,\nu, \qquad \nu=\frac{y}{\sqrt{2(n-k)}},
\]
we have
\[
g_\tau(\partial_\tau X,\nu_\tau) = \partial_\tau u\,g_\tau(\nu,\nu_\tau).
\]
In the Euclidean metric,
\[
\delta(\nu,\nu_E) = S^{-1}.
\]
Moreover,
\[
\left| g_\tau(\nu,\nu_\tau)-S^{-1} \right| \leq \left|(g_\tau-\delta)(\nu,\nu_\tau)\right| + \left|\delta(\nu,\nu_\tau-\nu_E)\right| \leq CRe^{-\frac{\tau}{2}}.
\]
Define 
$$\eta_\tau := g_\tau(\nu, \nu_\tau) - S^{-1}.$$
Then
\[
g_\tau(\partial_\tau X,\nu_\tau) = \left(S^{-1}+\eta_\tau\right)\partial_\tau u, \qquad |\eta_\tau| \leq CRe^{-\frac{\tau}{2}}.
\]
The rescaled mean curvature flow equation is
\[
(\partial_\tau X)^\perp_{g_\tau} = \left(H_\tau+\frac12 g_\tau(X,\nu_\tau) \right)\nu_\tau.
\]
Taking the $g_\tau$-inner product with $\nu_\tau$ gives
\[
g_\tau(\partial_\tau X,\nu_\tau) = H_\tau+\frac12g_\tau(X,\nu_\tau).
\]
Thus
\[
\left(S^{-1}+\eta_\tau\right)\partial_\tau u = H_\tau+\frac12g_\tau(X,\nu_\tau).
\]
Using
\[
H_\tau=H_E+(H_\tau-H_E),
\]
and
\[
g_\tau(X,\nu_\tau) = \sqrt{2(n-k)}+u-z_i\partial_i u + \big(g_\tau(X,\nu_\tau)-(\sqrt{2(n-k)}+u-z_i\partial_i u)\big),
\]
together with
\[
H_E+\frac12S^{-1}(\sqrt{2(n-k)}+u-z_i\partial_i u) = S^{-1}(Lu+Q(u)),
\]
we obtain
\[
\left(S^{-1}+\eta_\tau\right)\partial_\tau u = S^{-1}(Lu+Q(u)) +(H_\tau-H_E) + \frac12
\big(g_\tau(X,\nu_\tau)-(\sqrt{2(n-k)}+u-z_i\partial_i u)\big).
\]
Assume throughout that $|u|+|\nabla u|+|\nabla^2u|\leq 1.$ Then $S$ is uniformly bounded above and below. Thus, decreasing $r_0$ further if necessary, $S^{-1}+\eta_\tau = S^{-1}(1+S\eta_\tau)$ is uniformly bounded away from zero. Hence
\[
\frac{S^{-1}}{S^{-1}+\eta_\tau} = 1+\widetilde\eta_\tau, \qquad |\widetilde\eta_\tau| \leq CRe^{-\frac{\tau}{2}}.
\]
Dividing the preceding equation by $S^{-1}+\eta_\tau$, we obtain
\[
\partial_\tau u = Lu+Q(u)+\mathcal{E}_\tau,
\]
where
\[
\mathcal{E}_\tau = \widetilde\eta_\tau(Lu+Q(u)) + \frac{H_\tau-H_E}{S^{-1}+\eta_\tau} +
\frac{g_\tau(X,\nu_\tau)-(\sqrt{2(n-k)}+u-z_i\partial_i u)}{2(S^{-1}+\eta_\tau)}.
\]
Under the above bounds on $u$, we have
\[
|Lu+Q(u)|\leq C,
\]
and the previously established estimates give
\[
|H_\tau-H_E|\leq CRe^{-\frac{\tau}{2}}
\]
and
\[
\left|g_\tau(X,\nu_\tau)-(\sqrt{2(n-k)}+u-z_i\partial_i u)\right|\leq CR^2e^{-\frac{\tau}{2}}.
\]
Consequently,
\[
|\mathcal{E}_\tau| \leq CRe^{-\frac{\tau}{2}} + CR^2e^{-\frac{\tau}{2}} \leq CR^2e^{-\frac{\tau}{2}},
\]
for $R\geq 1$, after enlarging $C$ if necessary. Here $C$ is independent of $R$ and $\tau$.

The preceding computation was carried out with the cylinder as the reference hypersurface. We will also use the graphical representation when the reference hypersurface is itself evolving by the rescaled mean curvature flow. Let $\{N_\tau\}$ be a rescaled mean curvature flow, and suppose that $M_\tau$ is represented as a normal graph over $N_\tau$,
\[
X(p,\tau) = p+u(p,\tau)\nu_{N_\tau}(p).
\]
The evolution equation for the graph function is obtained in the same way as in the stationary case, with the geometric quantities computed with respect to $N_\tau$. Since $N_\tau$ remains uniformly close to the
cylinder, its second fundamental form and the corresponding zeroth-order coefficients are uniformly bounded. Thus, on $B_R(0)$, we obtain
\[
\left|\partial_\tau u -\Delta_{N_\tau}u +\frac12\langle X,\nabla_{N_\tau}u\rangle \right| \leq C\left(|u| + |\nabla_{N_\tau}u| + R^2e^{-\frac{\tau}{2}} \right),
\]
where $C$ is independent of $R$ and $\tau$. The final term $R^2e^{-\frac{\tau}{2}}$ arises from the perturbation of the ambient metric in the Fermi coordinates.

\subsection{Neumann Boundary Condition}\label{section 2.3} 

Suppose that $M_\tau$ is represented as a normal graph over the half-cylinder $\mathscr{C}_+$,
\[
X(p,\tau) = p+u(p,\tau)\nu_{\mathscr{C}_+}(p).
\]
The Euclidean unit normal to $M_\tau$ satisfies
\[
\nu_{M_\tau} = \nu_{\mathscr{C}_+} - \nabla_{\mathscr{C}_+}u + E,
\]
where, assuming $|u|,|\nabla_{\mathscr{C}_+}u|\leq 1,$ the remainder satisfies
\[
|E| \leq C\left(|u|\,|\nabla_{\mathscr{C}_+}u| + |\nabla_{\mathscr{C}_+}u|^2\right).
\]
Along $\partial\mathscr{C}_+$, the free boundary condition is
\[
g_\tau(\nu_{M_\tau},e_{n+1})=0.
\]
Since
\[
\delta(\nu_{\mathscr{C}_+},e_{n+1})=0
\]
along $\partial\mathscr{C}_+$, we obtain
\[
\delta(\nabla_{\mathscr{C}_+}u,e_{n+1}) = \delta(E,e_{n+1}) + (g_\tau-\delta)(\nu_{M_\tau},e_{n+1}).
\]
Using $|g_\tau-\delta| \leq Ce^{-\frac{\tau}{2}}|X|$ and $|X|\leq CR$ on $B_R(0)$, we obtain
\[
\left|\delta(\nabla_{\mathscr{C}_+}u,e_{n+1})\right| \leq C\left(|u|\,|\nabla_{\mathscr{C}_+}u| + |\nabla_{\mathscr{C}_+}u|^2 + Re^{-\frac{\tau}{2}} \right).
\]
Thus the free boundary condition gives a perturbed Neumann condition for the graph function $u$, with the last term arising from the perturbation of the ambient metric.

\subsection{Monotonicity Formula}\label{section 2.4} 

In this section we prove a monotonicity formula in Fermi coordinates. The proof is an adaptation of the argument in [\citealp{edelen2020free}, Section~5]. Write $\sigma=-t>0$. We set $\kappa = c_0(n)^{-1}r_0$ 
for a sufficiently large constant $c_0(n)$, and restrict to $-\beta_0(n)\kappa^2<t<0.$ Define
\[
\rho(x,t) = (4\pi\sigma)^{-\frac {n}{2}} \exp\left(-\frac{|x|^2}{4\sigma}\right).
\]
Let
\[
\eta(s)=(1-s)_+^4
\]
and define
\[
\phi(x,t) = \eta\left(\left(\frac{\kappa^2}{\sigma}\right)^{\frac34} \frac{|x|^2-\alpha\sigma}{\kappa^2}
\right),
\]
where $\alpha=\alpha(n)$ will be chosen in the proof of Theorem \ref{thm 2.1}. We denote by $d\mathcal{H}^n_g$ the $n$-dimensional Hausdorff measure induced by the metric $g$. When $g =\delta,$ we will just denote $d\mathcal{H}^n_\delta = d\mathcal{H}^n.$ Set
\[
\Phi  = \rho\phi.
\]
Since $\Phi$ depends only on $|x|^2$, it is even with respect to $S=\{x_{n+1}=0\}$. In particular,
\[
\partial_{x_{n+1}} \Phi = 0\qquad\text{on }S.
\]
We now verify that the support of $\Phi$ is contained in $B_{r_0}(0)$, so that all subsequent computations take place within the Fermi coordinate neighborhood. If $x\in\operatorname{supp}\phi(\cdot,t)$, then
\[
\left(\frac{\kappa^2}{\sigma}\right)^{\frac34} \frac{|x|^2-\alpha\sigma}{\kappa^2} \leq 1,
\]
and hence
\[
|x|^2 \leq \alpha\sigma + \kappa^2 \left(\frac{\sigma}{\kappa^2}\right)^{\frac34}.
\]
Since $\sigma\leq\beta_0(n)\kappa^2,$ we obtain $|x|\leq C(n)\kappa.$ Thus, after choosing $c_0(n)$ sufficiently large, $\operatorname{supp}\phi(\cdot,t) \subset B_{r_0}(0).$ Since, $\Phi = \rho \phi,$ we have $\operatorname{supp} \Phi\subset B_{r_0}(0),$ so all subsequent computations take place within the Fermi coordinate neighborhood. In what follows, $L=T_xM_t$ denotes the tangent plane to the Fermi-coordinate pullback, and $\pi_L^\perp$ denotes the Euclidean orthogonal projection onto $L^\perp$. 

\begin{thm} \label{thm 2.1}
There exist constants $\beta_0=\beta_0(n)>0$ and $A=A(n)>0$ such that, for $\kappa$ sufficiently small, if
\[
\mathcal E_0\geq \int_{M_{-\beta_0\kappa^2}}\phi\,d\mathcal H^n_g,
\]
then
\[
t\longmapsto e^{A(-t)^{\frac{1}{4}}} \int_{M_t} \Phi \,d\mathcal H^n_g + A\mathcal E_0(-t)
\]
is decreasing on $[-\beta_0\kappa^2,0)$.
\end{thm}

\begin{proof}
We first estimate the Gaussian factors. Since
\[
\rho(x,t) = (4\pi\sigma)^{-\frac n2} \exp\left(-\frac{|x|^2}{4\sigma}\right), \qquad \sigma=-t,
\]
we have
\[
D_i\rho=-\frac{x_i}{2\sigma}\rho, \qquad D_iD_j\rho = \left(-\frac{\delta_{ij}}{2\sigma} + \frac{x_ix_j}{4\sigma^2}\right)\rho.
\]
Moreover,
\[
\partial_t\rho = \left(\frac{n}{2\sigma} - \frac{|x|^2}{4\sigma^2}\right)\rho.
\]
Using the comparison between the Euclidean and $g$-Hessians, we obtain
\[
\left|\operatorname{tr}_{L,g}\nabla_g^2\rho - \operatorname{tr}_{L}D^2\rho\right| \leq C\left(\frac{|x|}{\sigma} + \frac{|x|^3}{\sigma^2}\right)\rho.
\]
Since
\[
\operatorname{tr}_{L}D^2\rho = \left(-\frac{n}{2\sigma} + \frac{|x|^2-|\pi_L^\perp x|^2}{4\sigma^2} \right)\rho,
\]
we conclude that
\[
\left|\left(\partial_t+\operatorname{tr}_{L,g}\nabla_g^2\right)\rho + \frac{|\pi_L^\perp x|^2}{4\sigma^2}\rho\right| \leq C\left(\frac{|x|}{\sigma} + \frac{|x|^3}{\sigma^2} \right)\rho.
\]
Next, we consider the cutoff function. The cutoff estimate was proved in [\citealp{edelen2020free}, Theorem ~4.9] for the Euclidean metric. We now show that, after choosing $\alpha=\alpha(n)$ sufficiently large and then $r_0=r_0(n)$ sufficiently small, the same estimate holds with respect to $g$, namely
\[
\left(\partial_t-\operatorname{tr}_{L,g}\nabla_g^2\right)\phi \leq0
\]
for $-\beta_0(n)\kappa^2<t<0$ on $B_{r_0}(0)$. Let
\[
\psi = 1 - \left(\frac{\kappa^2}{\sigma}\right)^{\frac34} \frac{|x|^2-\alpha\sigma}{\kappa^2}.
\]
On $\{\phi>0\}$, so that $\phi=\psi^4$, direct differentiation gives
\[
D\phi=-8\kappa^{-\frac12}\sigma^{-\frac34}\psi^3x
\]
and
\[
D^2\phi(v,v) = -8\kappa^{-\frac12}\sigma^{-\frac34}\psi^3|v|^2 + 48\kappa^{-1}\sigma^{-\frac32}\psi^2 \langle x,v\rangle^2.
\]
Since $\sigma=-t$, we also obtain, for $\alpha=\alpha(n)$ sufficiently large,
\[
(\partial_t-\operatorname{tr}_L D^2)\phi \leq -c\kappa^{-\frac12}\sigma^{-\frac34}\psi^3 - 48\kappa^{-1}\sigma^{-\frac32}\psi^2|\pi_Lx|^2.
\]
Now, let $e_1,\ldots,e_n$ be a $g$-orthonormal basis of $L$, and let $v_1,\ldots,v_n$ be a Euclidean orthonormal basis of $L$. Since $|g-\delta|\leq C|x|,$ after decreasing $r_0$ if necessary, we may choose the bases so that $|e_i-v_i|\leq C|x|$ for every $i$. In particular,
\[
\left|\sum_{i=1}^n |e_i|^2-n\right| \leq C|x|.
\]
Moreover,
\[
\sum_{i=1}^n\langle x,e_i\rangle^2 = \sum_{i=1}^n\langle x,v_i\rangle^2 + \sum_{i=1}^n \left(\langle x,e_i\rangle^2-\langle x,v_i\rangle^2\right).
\]
Since $e_i-v_i\in L$ and $|e_i-v_i|\leq C|x|$,
\[
|\langle x,e_i-v_i\rangle| = |\langle \pi_Lx,e_i-v_i\rangle| \leq C|x||\pi_Lx|.
\]
Therefore,
\[
\left|\sum_{i=1}^n\langle x,e_i\rangle^2 - |\pi_Lx|^2\right| \leq C|x||\pi_Lx|^2.
\]
Since $|\Gamma_g|\leq C$, we have
\[
\left|\sum_{i=1}^n D\phi\bigl(\Gamma_g(e_i,e_i)\bigr)\right| \leq C|D\phi|\leq C|x| \kappa^{-\frac12} \sigma^{-\frac34} \psi^3.
\]
Using the formula for $D^2\phi(v,v)$ together with the preceding estimates, we obtain
\[
\operatorname{tr}_{L,g}\nabla_g^2\phi \leq \operatorname{tr}_L D^2\phi + C|x|\kappa^{-\frac12}\sigma^{-\frac34}\psi^3 + C|x|\kappa^{-1}\sigma^{-\frac32} \psi^2|\pi_Lx|^2.
\]
Combining this with the Euclidean estimate and using $|x|\leq r_0$, we get
\[
(\partial_t-\operatorname{tr}_{L,g}\nabla_g^2)\phi \leq - (c-Cr_0)\kappa^{-\frac12}\sigma^{-\frac34}\psi^3 - (48-Cr_0)\kappa^{-1}\sigma^{-\frac32} \psi^2|\pi_Lx|^2.
\]
Choosing $r_0$ sufficiently small gives
\[
(\partial_t-\operatorname{tr}_{L,g}\nabla_g^2)\phi\leq0.
\]
We now apply the first variation formula. With our convention for the mean curvature vector, for every smooth function $u=u(x,t)$ with compact support in the Fermi coordinate neighborhood,
\[
\frac{d}{dt}\int_{M_t}u\,d\mathcal{H}^n_g = \int_{M_t} \left(\partial_tu + \langle\nabla_g u, H \rangle_g - |H|_g^2u \right)\,d\mathcal{H}^n_g.
\]
The free-boundary condition says that the conormal $\nu$ is normal to $S$. It follows immediately that $\langle \nabla_{M_t} \Phi, \nu \rangle_g =0$ on $\partial M_t,$ and the boundary term arising from integration by parts vanishes. Therefore
\[
\int_{M_t} \left(\phi\, \operatorname{tr}_{L,g}\nabla_g^2 \rho -
\rho\, \operatorname{tr}_{L,g}\nabla_g^2 \phi
\right)
=
-\int_{M_t}
\langle H,
\phi\, \nabla_{M_t}\rho
-
\rho\, \nabla_{M_t}\phi
\rangle_g.
\]
Since $\Phi = \rho\phi$, the first variation formula gives
\[
\begin{aligned}
\frac{d}{dt}\int_{M_t} \Phi\,d\mathcal{H}^n_g
={}&
\int_{M_t} \left(\partial_t(\rho\phi) + \langle H, \nabla_{M_t}(\rho\phi) \rangle_g\right)\,d\mathcal{H}^n_g
- \int_{M_t}|H|_g^2 \Phi \,d\mathcal{H}^n_g.
\end{aligned}
\]
Expanding the derivatives and using the integration-by-parts, we obtain
\[
\begin{aligned}
\frac{d}{dt}\int_{M_t} \Phi \,d\mathcal{H}^n_g =
{}& 
\int_{M_t} \phi \left(\partial_t + \operatorname{tr}_{L,g}\nabla_g^2\right)\rho\,d\mathcal{H}^n_g + \int_{M_t} \rho \left(\partial_t - \operatorname{tr}_{L,g}\nabla_g^2 \right)\phi\,d\mathcal{H}^n_g \\
&+ 
2\int_{M_t} \langle H,  \phi\nabla_{M_t}\rho \rangle_g \,d\mathcal{H}^n_g - \int_{M_t}|H|_g^2 \Phi\,d\mathcal{H}^n_g.
\end{aligned}
\]
We next use the Gaussian estimate. Since $\nabla_g\rho=-\frac{g^{-1}x}{2\sigma}\rho$ and $|g^{-1}x-x| \leq |g^{-1}-\delta|\,|x| \leq C|x|^2,$ we have
\[
\left|\nabla_g\rho+\frac{x}{2\sigma}\rho\right|_g \leq C\frac{|x|^2}{\sigma}\rho.
\]
Thus, in the term involving the mean curvature, replacing $\nabla_g\rho$ by its Euclidean leading term
$-\frac{x\rho}{2\sigma}$ produces an error bounded by $C|H|_g \phi \frac{|x|^2}{\sigma} \rho.$ Let $\pi_{L,g}^\perp$ denote the $g$-orthogonal projection onto the $g$-normal space of $L$. Since $H$ is normal to $L$ with respect to $g$,
\[
\langle H, x \rangle_g = \langle H, \pi_{L,g}^\perp x \rangle_g.
\]
Moreover, since $|g-\delta| \leq C|x|$,
\[
\left|\pi_{L,g}^\perp x-\pi_L^\perp x\right|\leq C|x|^2.
\]
Consequently,
\[
\left||\pi_{L,g}^\perp x|_g^2-|\pi_L^\perp x|^2\right| \leq C|x|^3.
\]
Thus the Gaussian estimate can be written as
\[
\left|\left(\partial_t+\operatorname{tr}_{L,g}\nabla_g^2\right)\rho + \frac{|\pi_{L,g}^\perp x|_g^2}{4\sigma^2}\rho \right| \leq C\left(\frac{|x|}{\sigma} + \frac{|x|^3}{\sigma^2}\right)\rho.
\]
Using the first variation identity and the cutoff inequality, we therefore obtain
\[
\frac{d}{dt}\int_{M_t} \Phi \,d\mathcal{H}^n_g \leq -\int_{M_t} \left|H+\frac{\pi_{L,g}^\perp x}{2\sigma}
\right|_g^2 \Phi \,d\mathcal{H}^n_g + C\int_{M_t}\left(\frac{|x|}{\sigma} + \frac{|x|^3}{\sigma^2} +
\frac{|H|_g|x|^2}{\sigma}\right) \Phi \,d\mathcal{H}^n_g.
\]
We now absorb the mean curvature error. By Young's inequality,
\[
C\frac{|H|_g|x|^2}{\sigma} \leq \frac12 \left|H+\frac{\pi_{L,g}^\perp x}{2\sigma}\right|_g^2 + C\frac{|x|^3}{\sigma^2} + C\frac{|x|^4}{\sigma^2}.
\]
Multiplying by $\Phi$ and using the preceding estimate, we obtain
\[
\frac{d}{dt}\int_{M_t} \Phi \,d\mathcal{H}^n_g \leq -\frac{1}{2}\int_{M_t} \left|H+\frac{\pi_{L,g}^\perp x}{2\sigma} \right|_g^2 \Phi \,d\mathcal{H}^n_g + C\int_{M_t} \left(\frac{|x|}{\sigma}
+ \frac{|x|^3}{\sigma^2} + \frac{|x|^4}{\sigma^2}\right) \Phi \,d\mathcal{H}^n_g.
\]
Since $|x|\leq r_0$ on the support of $\Phi$, 
\[
\frac{|x|^4}{\sigma^2} \leq C\frac{|x|^3}{\sigma^2}.
\]
Therefore,
\[
\frac{d}{dt}\int_{M_t} \Phi \,d\mathcal{H}^n_g \leq -\frac{1}{2}\int_{M_t} \left|H+\frac{\pi_{L,g}^\perp x}{2\sigma}\right|_g^2 \Phi \,d\mathcal{H}^n_g + C\int_{M_t} \left(\frac{|x|}{\sigma} + \frac{|x|^3}{\sigma^2}\right) \Phi \,d\mathcal{H}^n_g.
\]
Set
\[
q:=1+\frac{1}{2n+2}>1.
\]
Using the Gaussian expression for $\rho$, we have
\[
\frac{|x|}{\sigma}\rho \leq 1+ \left(\frac{|x|}{\sigma}\rho\right)^q.
\]
Moreover,
\[
\left(\frac{|x|}{\sigma}\rho\right)^q = C(n)\,\sigma^{-\frac{q}{2}-\frac{nq}{2}}\left(\frac{|x|}{2\sqrt{\sigma}} e^{-\frac{|x|^2}{4\sigma}} \right)^q.
\]
Using $y^\beta e^{-y} \leq C(\beta,\gamma)e^{-(1-\gamma)y},$ with $\gamma=1-q^{-1}$, we obtain
\[
\left(\frac{|x|}{\sigma}\rho\right)^q \leq C(n)\sigma^{-\frac{nq}{2}-\frac{q}{2}+\frac n2}\rho.
\]
By the choice of $q$,
\[
-\frac{nq}{2} - \frac{q}{2}+\frac n2 = -\frac{(n+1)q-n}{2} = -\frac34.
\]
Hence,
\[
\frac{|x|}{\sigma}\rho \leq 1+C(n)\sigma^{-\frac34}\rho.
\]
Similarly,
\[
\frac{|x|^3}{\sigma^2}\rho \leq 1+ \left(\frac{|x|^3}{\sigma^2}\rho\right)^q.
\]
As above,
\[
\left(\frac{|x|^3}{\sigma^2}\rho \right)^q \leq C(n) \sigma^{-\frac{nq}{2}+\frac{3q}{2}-2q+\frac n2}\rho.
\]
Since
\[
-\frac{nq}{2}+\frac{3q}{2}-2q+\frac n2 = -\frac{nq}{2}-\frac q2+\frac n2 = -\frac34,
\]
we obtain
\[
\frac{|x|^3}{\sigma^2}\rho \leq 1+C(n)\sigma^{-\frac34}\rho.
\]
Since $\Phi = \rho\phi$ and $0 \leq \phi \leq1$, it follows that
\[
\left(\frac{|x|}{\sigma} + \frac{|x|^3}{\sigma^2}\right) \Phi = \phi \left(\frac{|x|}{\sigma}\rho + \frac{|x|^3}{\sigma^2}\rho\right) \leq C\sigma^{-\frac34}\Phi + C\phi.
\]
Consequently,
\[
\frac{d}{dt}\int_{M_t} \Phi \,d\mathcal{H}^n_g \leq C\sigma^{-\frac34}\int_{M_t} \Phi \,d\mathcal{H}^n_g + C\int_{M_t}\phi\,d\mathcal{H}^n_g.
\]
Choose $A=A(n)>0$ sufficiently large. Since $\sigma=-t,$ we have $\frac{d}{dt}\sigma^{\frac14} = -\frac14\sigma^{-\frac34},$ and hence
\[
\frac{d}{dt} \left(e^{A\sigma^{\frac14}} \int_{M_t} \Phi \,d\mathcal{H}^n_g\right) = e^{A\sigma^{\frac14}}
\frac{d}{dt}\int_{M_t} \Phi \,d\mathcal{H}^n_g - \frac{A}{4}\sigma^{-\frac34} e^{A\sigma^{\frac14}} \int_{M_t} \Phi\,d\mathcal{H}^n_g.
\]
Thus, after choosing $A$ sufficiently large,
\[
\frac{d}{dt} \left(e^{A\sigma^{\frac14}} \int_{M_t} \Phi \,d\mathcal{H}^n_g \right)\leq A e^{A\sigma^{\frac14}} \int_{M_t}\phi\,d\mathcal{H}^n_g.
\]
Next, we estimate the cutoff mass. As above, the boundary term vanishes by the evenness of $\phi$ across $S$ and the free-boundary condition. Since
\[
\left(\partial_t-\operatorname{tr}_{L,g}\nabla_g^2\right)\phi \leq 0,
\]
the first variation formula gives
\[
\frac{d}{dt} \int_{M_t}\phi\,d\mathcal{H}^n_g \leq \int_{M_t} \left(\partial_t\phi + \langle H,\nabla_g\phi\rangle_g - |H|_g^2\phi \right)\,d\mathcal{H}^n_g.
\]
Using integration by parts,
\[
\int_{M_t} \langle H, \nabla_g\phi\rangle_g \;d\mathcal{H}^n_g = -\int_{M_t} \operatorname{tr}_{L,g}\nabla_g^2\phi\,d\mathcal{H}^n_g,
\]
and therefore
\[
\frac{d}{dt} \int_{M_t}\phi\,d\mathcal{H}^n_g \leq - \int_{M_t}|H|_g^2\phi\,d\mathcal{H}^n_g + \int_{M_t} \left(\partial_t-\operatorname{tr}_{L,g}\nabla_g^2\right)\phi\,d\mathcal{H}^n_g \leq0.
\]
Consequently,
\[
\int_{M_t}\phi\,d\mathcal{H}^n_g \leq \int_{M_{-\beta_0\kappa^2}}\phi\,d\mathcal{H}^n_g \leq \mathcal E_0.
\]
Since $\sigma\leq\beta_0\kappa^2$, we have
\[
e^{A\sigma^{\frac14}} \leq e^{A(\beta_0\kappa^2)^{\frac14}} \leq C(n).
\]
Hence, after enlarging $A=A(n)$ if necessary,
\[
\frac{d}{dt}\left(e^{A\sigma^{\frac14}}\int_{M_t} \Phi \,d\mathcal{H}^n_g\right) \leq A\mathcal E_0.
\]
Therefore,
\[
\frac{d}{dt} \left[e^{A\sigma^{\frac14}} \int_{M_t} \Phi \,d\mathcal{H}^n_g + A\mathcal E_0(-t)\right] \leq 0.
\]
Thus
\[
t\longmapsto e^{A(-t)^{\frac14}} \int_{M_t} \Phi \,d\mathcal{H}^n_g + A\mathcal E_0(-t)
\]
is decreasing on $[-\beta_0\kappa^2,0)$. This proves the desired monotonicity formula.
\end{proof}

\begin{remark}
If the Fermi coordinates are Euclidean, so that $g=\delta$, then one may take $A=0$, and no $\mathcal E_0$ is needed. In this case,
\[
\frac{d}{dt}\int_{M_t} \Phi \,d\mathcal H^n \leq -\int_{M_t} \left|H+\frac{x^\perp}{2\sigma}\right|^2 \Phi \,d\mathcal H^n,
\]
where $x^\perp$ denotes the Euclidean normal projection of $x$ onto $T_xM_t$.
\end{remark}

We next record the local mass estimate that will be used below. Since we have established both the cutoff evolution inequality
\[
\left(\partial_t-\operatorname{tr}_{L,g}\nabla_g^2\right)\phi\leq 0
\]
and the corresponding first variation formula for the smooth free-boundary flow, the proof of \cite[Corollary 4.10]{edelen2020free} applies verbatim in our setting, with the Euclidean metric replaced by the Fermi metric $g$. In particular, there exist constants
\[
\gamma=\gamma(n)\in(0,1),\qquad c_0=c_0(n)>0,\qquad C=C(n)>0
\]
such that, for $\kappa$ sufficiently small, whenever all relevant balls are contained in $B_{r_0}(0)$, we have
\[
\mathcal H_g^n\bigl(M_{t+s}\cap B_{\gamma\kappa}(y)\bigr) \leq C\mathcal H_g^n\bigl(M_t\cap B_\kappa(y)\bigr)
\]
for every $y\in B_{r_0}(0)$, $t\geq 0$, and $0\leq s\leq c_0\kappa^2$.

Since $g$ and the Euclidean metric are uniformly equivalent on $B_{r_0}(0)$, after possibly enlarging $C$, the same estimate holds with $\mathcal H_g^n$ replaced by the Euclidean measure $\mathcal H^n$.

We now pass from the preceding one-ball estimate to a fixed spatial ball. Let $B_r(z)\subset B_{r_0}(0)$ and choose a Vitali covering $\{B_{\gamma\kappa}(y_i)\}_i$ of $B_r(z)$ with $y_i\in B_r(z)$. The corresponding balls $B_\kappa(y_i)$ have uniformly bounded overlap. Hence
\[
\begin{aligned}
\mathcal H_g^n(M_{t+s}\cap B_r(z)) &\leq \sum_i \mathcal H_g^n(M_{t+s}\cap B_{\gamma\kappa}(y_i)) \\
&\leq C\sum_i \mathcal H_g^n(M_t\cap B_\kappa(y_i)) \\
&\leq C\,\mathcal H_g^n(M_t\cap B_{r+\kappa}(z)),
\end{aligned}
\]
for every $0\leq s\leq c_0\kappa^2$.

We now pass to rescaled coordinates. The purpose of the Fermi coordinates in the preceding argument was to flatten the free boundary hypersurface, which remains static under the rescaling. More precisely, set
\[
M_\tau=e^{\frac{\tau}{2}}M_{-e^{-\tau}}, \qquad S_\tau=e^{\frac{\tau}{2}}S=S=\{x_{n+1}=0\}.
\]
Recall that $g_\tau(x) = g(e^{-\frac{\tau}{2}}x)$ denote the rescaled metric on the rescaled Fermi coordinates and the corresponding area measure is denoted by $d\mathcal H^n_{g_\tau}$. Define
\[
\rho(x)=(4\pi)^{-\frac n2}e^{-\frac{|x|^2}{4}},
\]
and
\[
\phi(x,\tau) = \eta\left(\kappa^{-\frac12}e^{-\frac{\tau}{4}}\left(|x|^2-\alpha\right)\right), \qquad \eta(s)=(1-s)_+^4,
\]
and set
\[
\Phi(x,\tau)=\rho(x)\phi(x,\tau).
\]

\begin{thm}\label{thm 2.3}
There exist constants $\beta_0=\beta_0(n)>0$ and $A=A(n)>0$ such that, for $\kappa$ sufficiently small, if
\[
\mathcal E_0\geq e^{-\frac{n\tau_0}{2}} \int_{M_{\tau_0}}\phi(x,\tau_0)\, d\mathcal H^n_{g_{\tau_0}}, \qquad \tau_0=-\log(\beta_0\kappa^2),
\]
then
\[
\tau\longmapsto e^{A e^{-\frac{\tau}{4}}} \int_{M_\tau} \Phi(x,\tau)\,d\mathcal H^n_{g_\tau} + A\mathcal E_0e^{-\tau}
\]
is decreasing on $[\tau_0,\infty)$.

Moreover, if $\partial M_\tau\subset S_\tau=S=\{x_{n+1}=0\}$ and the corresponding coordinates are Euclidean, so that $g_\tau = \delta$, then one may take $A=0$, and no $\mathcal E_0$ is needed. More precisely, for any fixed $T$,
\[
\frac{d}{d\tau} \int_{M_\tau} \Phi(x,T+\tau) \,d\mathcal H^n \leq -\int_{M_\tau} \left|H+\frac{x^\perp}{2}\right|^2 \Phi(x, T+\tau) \,d\mathcal H^n.
\]
\end{thm}
We also record the corresponding volume growth estimate. We assume that the initial surface satisfies the local area-ratio bound
\[
\mathcal H^n(M_0\cap B_r(z))\leq C_0r^n
\]
for the relevant $z$ and $r$. By the local mass estimate, this bound propagates forward in time. Therefore, after parabolic rescaling,
\[
\mathcal H^n_{g_\tau}(M_\tau\cap B_R(0)) \leq CR^n, \qquad 1\leq R\leq r_0e^{\frac{\tau}{2}}.
\]
Since $g_\tau$ is uniformly equivalent to the Euclidean metric, after possibly increasing $C$,
\[
\mathcal H^n(M_\tau\cap B_R(0)) \leq CR^n, \qquad 1\leq R\leq r_0e^{\frac{\tau}{2}}.
\]
We now define the localized Gaussian density associated with the monotonicity formula. Since the Fermi coordinates are only defined in a fixed neighborhood of the boundary point under consideration, the definition is local.

Throughout the remainder of the paper, we assume that all hypersurfaces under consideration satisfy a uniform local area ratio bound by $\lambda_0$. All constants appearing below are allowed to depend on $\lambda_0$ in addition to the other fixed parameters. 

Fix a spacetime point
\[
X_0=(x_0,t_0),\qquad x_0\in B_{\frac{r_0}{2}}(0),\qquad -1<t_0<0,
\]
and write
\[
\sigma=t_0-t,\qquad z=x-x_0.
\]
Let $x_0^*=(x_0',-x_{0,n+1})$ denote the reflection of $x_0$ across $S=\{x_{n+1}=0\}$. Set
\[
G_0=g(x_0).
\]
We denote by $\langle\cdot,\cdot\rangle_{G_0}$ and $|\cdot|_{G_0}$ the inner product and norm induced by the constant metric $G_0$.

For $\sigma=t_0-t>0$, define the frozen-metric Gaussian
\[
\rho_{X_0,G_0}(x,t) = (4\pi\sigma)^{-\frac n2} \exp\left(-\frac{|x-x_0|_{G_0}^2}{4\sigma}\right),
\]
and the corresponding cutoff
\[
\phi_{X_0,G_0}(x,t) = \eta\left(\left(\frac{\kappa^2}{\sigma}\right)^{\frac34} \frac{|x-x_0|_{G_0}^2-\alpha\sigma}{\kappa^2}\right), \qquad \eta(s)=(1-s)_+^4.
\]
Set
\[
\Phi_{X_0,G_0} = \rho_{X_0,G_0}\phi_{X_0,G_0}.
\]
For the reflected point $X_0^*=(x_0^*,t_0),$ we define $\Phi_{X_0^*,G_0}$ using the same frozen metric $G_0$.

We define the localized Gaussian density by
\[
\Theta_g(M,X_0,r) =
\begin{cases}
\displaystyle \int_{M_{t_0-r^2}} \left(\Phi_{X_0,G_0} + \Phi_{X_0^*,G_0} \right)(x,t_0-r^2)\,d\mathcal H^n_g, & |x_{0,n+1}|<\dfrac{\kappa}{10}, \\[2ex]
\displaystyle \int_{M_{t_0-r^2}} \Phi_{X_0,G_0}(x,t_0-r^2)\, d\mathcal H^n_g, & |x_{0,n+1}|\geq\dfrac{\kappa}{10}.
\end{cases}
\]
For sufficiently small $\beta_0$, the reflected cutoff is supported in $\{x_{n+1}<0\}$ when $|x_{0,n+1}| = \frac{\kappa}{10}$, so $\Phi_{X_0^*,G_0}=0$ on $M_t$ and the two definitions agree near $|x_{0,n+1}| = \frac{\kappa}{10}$. We now state the corresponding localized monotonicity formula.

\begin{thm}\label{thm 2.4}
There exist constants $A=A(n)>0$ and $\beta_0=\beta_0(n)>0$ such that the following holds.

Let
\[
X_0=(x_0,t_0),\qquad x_0\in B_{\frac{r_0}{2}}(0),\qquad -1<t_0<0,
\]
and suppose that the flow is defined on the time interval $[t_*,t_0]$ where $t_0-\beta_0\kappa^2 \leq t_* < t_0.$ Suppose that $\mathcal E_0$ satisfies
\[
\mathcal E_0 \geq \int_{M_{t_*}} \phi_{X_0,G_0}(x,t_*)\,d\mathcal H_g^n.
\]
Then
\[
r\longmapsto e^{A\sqrt r}\Theta_g(M,X_0,r) + A\mathcal E_0r^2
\]
is increasing for
\[
0<r\leq \min\left\{\sqrt{t_0 - t_*},\sqrt{t_0+1}\right\}.
\]
\end{thm}

\begin{proof}
Set $t \in [t_*,t_0),$ and let
\[
\sigma=t_0-t,\qquad r=\sqrt{\sigma},\qquad z=x-x_0, \qquad G_0=g(x_0),
\]
and write
\[
\Phi_0 = \Phi_{X_0,G_0},\quad \phi_0=\phi_{X_0,G_0}, \qquad \Phi_0^* = \Phi_{X_0^*,G_0}, \quad \phi_0^*=\phi_{X_0^*,G_0}.
\]
The proof follows the argument of Theorem~\ref{thm 2.1}, with the Gaussian and cutoff centered at $x_0$ and with the metric frozen at $G_0 = g(x_0)$. Since $G_0$ is constant, the frozen-metric Gaussian calculation is reduced to the Euclidean one by a linear change of coordinates. On the relevant supports,
\[
|x - x_0|\leq C\kappa,\qquad |g - G_0|\leq C\kappa, \qquad |Dg|\leq C.
\]
In the reflected case $|x_{0,n+1}|<\frac{\kappa}{10}$, we also have $|x-x_0|\leq C\kappa$ on the support of $\Phi_0^*$. Thus all metric and cutoff error estimates are uniformly controlled as in Theorem~\ref{thm 2.1}. After choosing $\beta_0$ sufficiently small, we obtain
\[
(\partial_t-\operatorname{tr}_{L,g}\nabla_g^2)\phi_0\leq0, \qquad (\partial_t \operatorname{tr}_{L,g} \nabla_g^2)\phi_0^*\leq0,
\]
and the corresponding Gaussian estimate
\[
\frac{d}{dt}\int_{M_t} \Phi \,d\mathcal H_g^n \leq C\sigma^{-\frac34} \int_{M_t}\Phi \,d\mathcal H_g^n + C\int_{M_t} \phi\,d\mathcal H_g^n
\]
for each kernel whenever its boundary term vanishes. Suppose first that $|x_{0,n+1}|<\frac{\kappa}{10}.$ Set
\[
F = \Phi_0+ \Phi_0^*,\qquad \varphi = \phi_0+\phi_0^*.
\]
In Fermi coordinates, throughout the Fermi neighborhood,
\[
g_{\alpha,n+1}=0,\qquad g_{n+1,n+1}=1.
\]
Hence the constant metric $G_0=g(x_0)$ has the same block form, and reflection across $S$ preserves the $G_0$-norm in the tangential variables. Since $x_0^*$ is the reflection of $x_0$,
\[
F(x',-x_{n+1},t) = F(x',x_{n+1},t), \qquad \varphi(x',-x_{n+1},t)=\varphi(x',x_{n+1},t).
\]
Thus
\[
\partial_{x_{n+1}}F=\partial_{x_{n+1}}\varphi=0 \qquad\text{on }S.
\]
The free-boundary condition therefore makes the boundary terms in the tangential integrations by parts vanish. Adding the two Gaussian estimates gives
\[
\frac{d}{dt}\int_{M_t}F\,d\mathcal H_g^n \leq C\sigma^{-\frac34}\int_{M_t}F\,d\mathcal H_g^n + C\int_{M_t} \varphi\,d\mathcal H_g^n.
\]
Moreover, using the cutoff inequality above, the first variation formula, and the same integration by parts as in Theorem~\ref{thm 2.1},
\[
\begin{aligned}
\frac{d}{dt}\int_{M_t}\varphi\,d\mathcal H_g^n &\leq \int_{M_t} \left(\partial_t\varphi + \langle H,\nabla_g\varphi \rangle_g -|H|_g^2\varphi \right)d\mathcal H_g^n\\ 
&= \int_{M_t} (\partial_t-\operatorname{tr}_{L,g}\nabla_g^2)\varphi\,d\mathcal H_g^n - \int_{M_t} |H|_g^2\varphi \,d\mathcal H_g^n \leq 0.
\end{aligned}
\]
Consequently,
\[
\int_{M_t}\varphi\,d\mathcal H_g^n\leq\mathcal E_0.
\]
Now suppose $|x_{0,n+1}|\geq\frac{\kappa}{10}.$ Choosing $\beta_0$ sufficiently small, the cutoff support satisfies
\[
\operatorname{supp} \Phi_0\subset B_{\frac{\kappa}{20}}(x_0),
\]
and therefore is disjoint from $S$. There is therefore no boundary term, and the same argument gives
\[
\frac{d}{dt}\int_{M_t}\Phi_0\,d\mathcal H_g^n \leq C\sigma^{-\frac34}\int_{M_t}\Phi_0\,d\mathcal H_g^n + C\int_{M_t}\phi_0\,d\mathcal H_g^n,
\]
while
\[
\frac{d}{dt}\int_{M_t}\phi_0\,d\mathcal H_g^n\leq0.
\]
Thus in both cases, with
\[
F_{X_0,G_0} =
\begin{cases}
\Phi_0+\Phi_0^*,& |x_{0,n+1}|<\frac{\kappa}{10},\\
\Phi_0,& |x_{0,n+1}|\geq\frac{\kappa}{10},
\end{cases}
\]
and
\[
\varphi_{X_0,G_0} =
\begin{cases}
\phi_0+\phi_0^*,& |x_{0,n+1}|<\frac{\kappa}{10},\\
\phi_0,& |x_{0,n+1}|\geq\frac{\kappa}{10},
\end{cases}
\]
we have
\[
\frac{d}{dt}\int_{M_t}F_{X_0,G_0}\,d\mathcal H_g^n \leq C\sigma^{-\frac34}\int_{M_t}F_{X_0,G_0}\,d\mathcal H_g^n + C\mathcal E_0.
\]
Since $\frac{d}{dt}\sigma^{\frac14}=-\frac14\sigma^{-\frac34},$ choosing $A=A(n)$ sufficiently large yields
\[
\frac{d}{dt} \left(e^{A\sigma^{\frac14}} \int_{M_t}F_{X_0,G_0}\,d\mathcal H_g^n\right) \leq A\mathcal E_0.
\]
Putting $\sigma=r^2$ and $t=t_0-r^2$, we obtain
\[
\frac{d}{dr} \left(e^{A\sqrt r}\int_{M_{t_0-r^2}}F_{X_0,G_0}\,d\mathcal H_g^n \right) \geq-2A\mathcal E_0r.
\]
Hence
\[
\frac{d}{dr}\left[e^{A\sqrt r}\int_{M_{t_0-r^2}}F_{X_0,G_0}\,d\mathcal H_g^n + A\mathcal E_0r^2 \right] \geq 0.
\]
This finishes the proof.
\end{proof}

The following local regularity statement is the analogue of the local regularity theorem \cite[Theorem~8.1]{edelen2020free} (see also \cite{white2005local}) in the Fermi-coordinate setting. We use the localized Gaussian density $\Theta_g$. The proof follows the blow-up argument of \cite{edelen2020free}, with the Euclidean metric replaced by the Fermi metric $g$. More precisely, under a sequence of parabolic blow-ups, we consider the corresponding rescaled flows $M^i$ with respect to the rescaled metrics $g^i$. After passing to a subsequence, the metrics $g^i$ converge locally smoothly to a constant positive-definite metric $G_\infty$. After a fixed linear change of coordinates, $G_\infty$ becomes the Euclidean metric, while the flattened free boundary remains a hyperplane. Depending on the behavior of the rescaled distance to the free boundary, the limiting flow is either an ordinary Euclidean mean-curvature flow in $\mathbb{R}^{n+1}$, or a Euclidean mean-curvature flow in a half-space with free boundary on the limiting hyperplane. Thus the blow-up argument reduces to the corresponding Euclidean interior or free-boundary argument. We omit the proof.

\begin{thm}\label{thm 2.5}
There exist constants $\epsilon=\epsilon(n,r_0)>0$ and $C=C(n,r_0)<\infty$ such that the following holds.

Let \(M=(M_t)_{t\in(-1,0)}\) be a smooth, properly embedded \(n\)-dimensional mean curvature flow in \(\mathbb R^{n+1}\) with free boundary on a smooth hypersurface \(S\), with Fermi coordinates defined
on \(B_{r_0}(0)\), and let \(U\subset B_{r_0}(0)\). Let $X_0=(x_0,t_0)\in U\times(-1,0)$ and $0<R\leq c(n)^{-1}r_0$ satisfy \(t_0-R^2>-1\). Suppose that
\[
\Theta_g(M,X,r)<1+\epsilon
\]
for every \(X\in U\times(-1,0)\) and \(0<r<R\). Then
\[
|A|\leq \frac{C}{R}
\]
on \(M_t\cap B_R(x_0)\) for all $t_0-R^2<t<t_0$.
\end{thm}

The following pseudolocality statement is the analogue, in the present setting, of the graphical persistence result of Ilmanen-Neves-Schulze \cite{ilmanen2019short}. It says that an initially small graph remains a controlled graph on a smaller spatial scale for a short time, assuming a uniform bound on the area ratios. In the present setting, we work in Fermi coordinates, where the ambient Euclidean metric is represented by the metric $g$, and use the local regularity theorem above, formulated in terms of the localized Gaussian density $\Theta_g$, to obtain the corresponding graphical persistence result.

\begin{thm}\label{thm 2.6}
Fix $\eta>0$. Then there exist $r_1, r_2, \delta > 0$ with $r_2\ll r_1\ll r_0$, such that the following holds for every $y$ with $B_{2r_1}(y)\subset B_{r_0}(0)$. If $t_1 + r_2^2 < 0$ and $M_{t_1}$ is a normal graph over $\mathscr{C}_{t_1} \cap \mathbb{R}^{n+1}_+ \cap B_{r_1}(y)$ with $C^2$ norm at most $\delta$, then, for every $t\in[t_1,t_1+r_2^2]$, $M_t\cap B_{r_2}(y)$ remains a normal graph over $\mathscr{C}_t\cap\mathbb{R}^{n+1}_+\cap B_{r_2}(y)$, with $C^2$ norm at most $\eta$.
\end{thm}

\begin{proof}
We first establish the required Gaussian density estimate using the initial slice. Let $x\in B_{r_2}(y)$ and $0<r\leq r_2$, and set $X=(x,t_1+r^2)$. We have $t_1 + r^2 < t_1 + r_2^2 < 0,$ and the time slice appearing in the definition of $\Theta_g(M,X,r)$ is precisely $M_{t_1}$. Thus
\[
\Theta_g(M,X,r) = \int_{M_{t_1}} \left(\Phi_{X,G_0} + \Phi_{X^*,G_0} \right)\,d\mathcal H_g^n,
\]
when $X$ lies in the boundary region, with the analogous expression involving only $\Phi_{X,G_0}$ when $X$ lies away from the boundary. Splitting the localized Gaussian density into the contribution from $B_{r_1}(y)$ and its complement, we obtain
\[
\Theta_g(M,X,r) = \int_{M_{t_1}\cap B_{r_1}(y)} \left(\Phi_{X,G_0} + \Phi_{X^*,G_0} \right)\,d\mathcal H_g^n + \int_{M_{t_1}\setminus B_{r_1}(y)}\left(\Phi_{X,G_0} + \Phi_{X^*,G_0} \right)\,d\mathcal H_g^n,
\]
with the reflected terms omitted in the interior case.

On $M_{t_1}\cap B_{r_1}(y)$ we write
\[
F(p)=p + u(p) \nu(p), \qquad p\in\mathscr{C}_{t_1} \cap \mathbb{R}^{n+1}_+\cap B_{r_1}(y).
\]
Since $u$ is a normal graph and $\|u\|_{C^2}\leq\delta$, the graphical Jacobian satisfies 
\[
J_gF\leq 1+C\delta.
\]
Moreover, the first and second derivatives of $F$ are uniformly close to those of the inclusion of $\mathscr{C}_{t_1} \cap \mathbb{R}^{n+1}_+$, and hence the Gaussian integral over the graph can be compared with the corresponding Gaussian integral over $\mathscr{C}_{t_1} \cap \mathbb{R}^{n+1}_+$. By choosing $\delta$ sufficiently small, we obtain
\[
\int_{M_{t_1}\cap B_{r_1}(y)} \Phi_{X,G_0}\,d\mathcal H_g^n \leq (1+C\delta) \int_{\mathscr{C}_{t_1} \cap \mathbb{R}^{n+1}_+} \Phi_{X,G_0}\,d\mathcal H_g^n + C\exp\left(-\frac{(r_1-r_2)^2}{Cr^2}\right).
\]
The same estimate holds for the reflected kernel in the boundary region, by the reflection argument used in the proof of Theorem~\ref{thm 2.4}.

The contribution from $M_{t_1}\setminus B_{r_1}(y)$ is estimated using the Gaussian decay. Since $x\in B_{r_2}(y), |z-x|\geq r_1-r_2$ for $z\in M_{t_1}\setminus B_{r_1}(y)$, and hence, by the local area-ratio bound,
\[
\int_{M_{t_1}\setminus B_{r_1}(y)} \left(\Phi_{X,G_0} + \Phi_{X^*,G_0}\right)\,d\mathcal H_g^n \leq C\exp\left(-\frac{(r_1-r_2)^2}{Cr^2}\right).
\]
Thus, after choosing $\frac{r_2}{r_1}$ sufficiently small, this contribution can be made smaller than $\frac{\epsilon}{8}$.

It remains to compare with the reference cylinder. Since $t_1+r_2^2<0$, after choosing $r_2$ sufficiently small, the reference cylinders $\mathscr{C}_t$ remain uniformly regular on the relevant time interval. Thus, on scales $r\le r_2$, the localized Gaussian density of $\mathscr{C}_{t_1}\cap\mathbb{R}^{n+1}_+$ is uniformly close to the density of its tangent model, with the error tending to zero as $r_2 \to 0$. Hence
\[
\Theta_g(\mathscr{C}_{t_1}\cap\mathbb{R}^{n+1}_+,X,r) \le 1+\frac{\epsilon}{4}.
\]
Combining the preceding estimates, and choosing first $\frac{r_2}{r_1}$ sufficiently small and then $\delta$ sufficiently small, we obtain
\[
\Theta_g(M,X,r)<1+\frac{\epsilon}{2}
\]
for every $x\in B_{r_2}(y)$ and $0<r\leq r_2$.

We now propagate this estimate to the spacetime region using the almost-monotonicity formula in Theorem~\ref{thm 2.4}. Let
\[
X=(x,t), \qquad x\in B_{r_2}(y), \qquad t\in(t_1,t_1 + r_2^2], \qquad t<0.
\]
Set $\rho=\sqrt{t-t_1}.$ Then $\rho \leq \,r_2.$ For every $0 < r \leq \rho$, Theorem~\ref{thm 2.4} gives
\[
e^{A\sqrt r}\Theta_g(M,X,r)+A\mathcal E_0r^2 \leq e^{A\sqrt{\rho}}\Theta_g(M,X,\rho) +A\mathcal E_0\rho^2.
\]
Since $t = t_1 +\rho^2$ and $\rho \leq r_2,$ the initial-slice estimate gives
\[
\Theta_g(M,X,\rho)<1+\frac{\epsilon}{2}.
\]
Choosing $r_2$ sufficiently small in the almost-monotonicity inequality, we obtain
\[
\Theta_g(M,X,r)<1+\epsilon.
\]
Thus the density hypothesis of Theorem~\ref{thm 2.5} is satisfied in the corresponding parabolic neighborhood, and hence the curvature estimate of Theorem~\ref{thm 2.5} gives a uniform bound for $|A_{M_t}|$ on a smaller parabolic neighborhood.

Since the reference flows $\mathscr{C}_t \cap \mathbb{R}^{n+1}_+$ have uniformly bounded curvature on the time interval, the initial $C^1$-smallness and the curvature bound imply that, after decreasing $\delta$ and $r_2$, the normal distance from $M_t$ to $\mathscr{C}_t \cap \mathbb{R}^{n+1}_+$ remains within the Fermi tubular neighborhood and the normal projection onto $\mathscr{C}_t \cap \mathbb{R}^{n+1}_+$ is one-to-one on the relevant component. Thus $M_t$ can be written as a normal graph over $\mathscr{C}_t \cap \mathbb{R}^{n+1}_+$. The standard relation between the second fundamental form of a normal graph and its graph function then gives a uniform $C^2$ bound for the graph function. By choosing $\delta$ and $r_2$ sufficiently small, this bound is at most $\eta$, and hence the $C^2$ norm of the graph remains at most $\eta$ throughout the stated time interval.
\end{proof}

The same argument also applies when $M_t$ is a small graph over a smooth mean curvature flow, provided the reference flow remains in the Fermi-coordinate neighborhood and has uniformly controlled geometry on the relevant parabolic region.

\subsection{Jacobi Fields}\label{section 2.5} 

We now describe the Jacobi fields on the free boundary half-cylindrical self-shrinkers. Let $\mathscr{C}_+$ denote either of the two half-cylindrical self-shrinkers, and let $\eta$ denote the outward unit conormal to $\partial\mathscr{C}_+$ in $\mathscr{C}_+$. Along $\partial\mathscr{C}_+$, $\eta$ is normal to the supporting hyperplane $S=\{x_{n+1}=0\}$. We define
\[
K_+ = \left\{u\in W^{2,2}(\mathscr{C}_+): Lu=0\text{ on }\mathscr{C}_+,\quad \frac{\partial u}{\partial\eta}=0 \text{ on }\partial\mathscr{C}_+ \right\}.
\]
It follows from \cite[Lemma~3.25]{colding2015uniqueness} and \cite[Proposition~3.1, Corollary~3.6]{colding2019regularity} that the kernel of the Jacobi operator on $\mathscr{C}$ decomposes as
\[
K=K_1\oplus K_2,
\]
where $K_1$ is generated by the rotational Jacobi fields and $K_2$ is generated by the quadratic fields. More precisely,
\[
K_1 = \operatorname{span} \left\{z_i\phi_\alpha(y) \right\},
\]
where $\{\phi_\alpha\}$ is a basis of the first spherical harmonics, and
\[
K_2 = \operatorname{span} \left\{z_i z_j-2\delta_{ij}: 1\leq i\leq j\leq k \right\}.
\]
Reflection across $\partial\mathscr{C}_+$ defines an isometric involution
\[
\mathcal{R}:\mathscr{C}\to\mathscr{C},
\]
whose fixed-point set is $\partial\mathscr{C}_+$. Since the Jacobi operator is invariant under $\mathcal{R}$, a Jacobi field on $\mathscr{C}_+$ satisfying the Neumann boundary condition extends evenly across $\partial\mathscr{C}_+$. Conversely, the restriction of an even Jacobi field on $\mathscr{C}$ satisfies the Neumann boundary condition on $\partial\mathscr{C}_+$. Hence
\[
K_+ \cong \left\{v\in K: v\circ\mathcal{R}=v \right\}.
\]
We therefore determine $K_+$ by imposing the evenness condition on the two components of $K$.

\medskip

\noindent
\textbf{(i) The axis is tangent to the boundary.}

In this case,
\[
\mathscr{C}_+ = \mathbb{S}^{n-k}_{\sqrt{2(n-k)}}\times\mathbb{R}^k_+,
\]
with boundary given by $z_k=0$. The reflection is
\[
z_k\longmapsto -z_k.
\]
Thus the rotational fields $z_i\phi_\alpha(y)$ are even precisely when $i<k$, and hence
\[
K_1^+ = \operatorname{span} \left\{z_i\phi_\alpha(y): 1\leq i\leq k-1 \right\}.
\]
For the quadratic fields, all modes with $i,j<k$ are even, as is the mode $z_k^2-2$. The mixed modes $z_i z_k$, $i<k$, are odd. Therefore
\[
K_2^+ = \operatorname{span} \left\{z_i z_j-2\delta_{ij}: 1\leq i\leq j\leq k-1 \right\} \oplus \operatorname{span}\{z_k^2-2\}.
\]
Consequently,
\[
K_+ = K_1^+\oplus K_2^+.
\]

\medskip

\noindent
\textbf{(ii) The axis is orthogonal to the boundary.}

In this case,
\[
\mathscr{C}_+ = \mathbb{S}^{n-k}_{+,\sqrt{2(n-k)}}\times\mathbb{R}^k,
\]
where $y_{n-k+1}\geq0.$ The reflection is
\[
y_{n-k+1}\longmapsto -y_{n-k+1}.
\]
Since the quadratic fields are independent of $y$, they are all even, and therefore
\[
K_2^+=K_2.
\]
The first spherical harmonics are generated by $y_1,\ldots,y_{n-k+1}.$ Now the functions $y_1,\ldots,y_{n-k}$ are even under the reflection, while $y_{n-k+1}$ is odd. Hence
\[
K_1^+ = \operatorname{span} \left\{z_i y_a: 1\leq i\leq k,\; 1\leq a\leq n-k \right\}.
\]
Thus
\[
K_+ = K_1^+\oplus K_2.
\]
In both cases, the free boundary Jacobi fields are obtained by restricting the full-cylinder Jacobi fields to the even subspace determined by the reflection across the boundary.

\subsection{The Comparison Flows}\label{section 2.6} 

We now adapt the construction of the rescaled flow \cite[Section~4]{ghosh2025cylindrical} to the free-boundary setting. The argument is the same as in the cylindrical setting, so we only indicate the modifications needed to incorporate the Neumann boundary condition. We denote by $L^2$ and $W^{k,2}$ the weighted Sobolev spaces with respect to $\rho$, and write $H^k=W^{k,2}$.

Let $\mathscr C_+$ denote either of the two half-cylinders described above, and let $\eta$ denote the outward unit conormal to $\partial\mathscr C_+$. All the notation below is understood with respect to the chosen half-cylinder. In particular, the space $K_2^+$ and its dimension may depend on the configuration. Set $m:=\dim K_2^+.$ We denote
\[
u_\alpha=\sum_{j=1}^m\alpha_jU_j, \qquad \alpha = (\alpha_1,\ldots,\alpha_m)\in\mathbb R^m,
\]
where $\{U_1,\ldots,U_m\}$ is the fixed orthonormal basis of $K_2^+$ chosen above. We impose the homogeneous Neumann condition
\[
\partial_\eta u=0 \qquad\text{on }\partial\mathscr C_+.
\]
Let $L_N$ denote the Neumann realization of $L$ on $\mathscr C_+$, with domain
\[
\operatorname{Dom}(L_N) = \left\{u\in H^2(\mathscr C_+): \partial_\eta u=0 \text{ on }\partial\mathscr C_+ \right\}.
\]
The estimates in Lemmas~4.2 and 4.3 carry over to the Neumann realization $L_N$ with the same bounds. Consequently, the proof of Proposition~4.4 applies with $L$ replaced by $L_N$ and with the nonlinear construction restricted to the Neumann class. Thus, for $R$ sufficiently large and $\delta R^6\leq C_\ell^{-1},$ there exists a map
\[
N_R: \ker L_N\cap \{u:\|u\|_{L^2}\leq\delta\} \longrightarrow C^{2+\ell,\alpha}(\mathscr C_+\cap B_R(0))
\]
such that
\[
\begin{aligned} 
\|N_R(u)\|_{C^{2+\ell,\alpha}(\mathscr C_+\cap B_R(0))} &\leq C_\ell R^6\|u\|_{L^2}^2,\\
P_{(\ker L_N)^\perp} Q\bigl(\chi_R(u+N_R(u))\bigr) &=L_N(N_R(u)),\\
\partial_\eta N_R(u)&=0 \qquad\text{on }\partial\mathscr C_+.
\end{aligned}
\]
We can fix a sufficiently small constant $c_1>0$. For each $\alpha$, define $R_\alpha$ by
\[
|\alpha|^{\frac94}e^{\frac{R_\alpha^2}{8}}=c_1,
\]
and set $N= N_{R_\alpha}$. Then, as in the cylindrical construction \cite{ghosh2025cylindrical}, there exists a solution $\alpha(\tau)$ with $\alpha(0) = \alpha$ defined for $0\leq\tau\leq c_0|\alpha|^{-1}$, which satisfies
\[
\partial_\tau u_{\alpha(\tau)} + L_N(N(u_{\alpha(\tau)})) = Q\bigl(\chi_{R_\alpha}(u_{\alpha(\tau)}+N(u_{\alpha(\tau)}))\bigr).
\]
We can then imitate the proof as in \cite{ghosh2025cylindrical} of the construction of the rescaled mean curvature flow which is graphical over $\mathscr C_+$ on $\mathscr C_+\cap B_{R_\alpha-1}(0)$. The function $u_{\alpha(\tau)}+N(u_{\alpha(\tau)})$ serves as the leading-order approximation. Let
\[
\mathcal P_N:L^2(\mathscr C_+)\longrightarrow L^2(\mathscr C_+)
\]
denote the $L^2$-orthogonal projection onto the subspace spanned by the Neumann eigenfunctions corresponding to eigenvalues $\leq \max\left\{-\frac12,-\frac1k\right\}.$ Thus $\mathcal P_N$ is the free-boundary counterpart of the spectral projection used in the cylindrical construction. The parabolic correction is constructed exactly as in Lemma~4.6, with $L$ and $\mathcal P$ replaced by $L_N$ and $\mathcal P_N$ respectively. The proof is unchanged, since all the required estimates above hold in the Neumann class. We may therefore choose $T_\alpha=|\alpha|^{-\frac18}.$ The fixed-point argument as in \cite{ghosh2025cylindrical} then gives, for $|\alpha|$ sufficiently small, a fixed point $v_\alpha$ of the map $\mathcal{F}_\alpha(v) := U_\alpha(v)$, where $U_\alpha(v)$ solves
\begin{equation*}
\begin{cases}
(\partial_\tau - L_N) U_{\alpha}(v) = F_{\alpha}(v) , \\
U_{\alpha}(v)(0) = -\int_0^{T_\alpha} e^{-sL_N} (1- \mathcal{P}_N) F_{\alpha}(v)(s) ds. &
\end{cases}
\end{equation*}
where
$$F_{\alpha}(v)(\tau) := Q(\chi_{R_\alpha} (u_{\alpha(\tau)} + N(u_{\alpha(\tau)}) + v)) - Q(\chi_{R_\alpha} (u_{\alpha(\tau)
} + N(u_{\alpha(\tau)}))) - \partial_\tau (\chi_{R_\alpha} N(u_{\alpha(\tau)})).$$
where $\alpha(0) = \alpha$, and $e^{\tau L_N}$ denotes the heat semigroup for the operator $L_N$. The parabolic correction is chosen in the Neumann class, so
\[
\partial_\eta v_\alpha(\tau)=0 \qquad\text{on }\partial\mathscr C_+.
\]
Moreover, for some large $\ell, \|v_\alpha\|_{H^\ell} \leq |\alpha|^{\frac52}$ and $\ell$ is large enough so that we can use Gaussian Sobolev embedding to get $\|\chi_{R_\alpha} v_\alpha\|_{C^{2,\alpha}} \leq C_\ell e^{\frac{R_\alpha^2}{8}} \|v_\alpha\|_{H^\ell} \leq C_\ell c_1|\alpha|^{\frac14}.$
Consequently, the graph of
\[
u_{\alpha(\tau)} + N(u_{\alpha(\tau)}) + v_\alpha(\tau), \qquad \alpha(0)=\alpha,
\]
over $\mathscr C_+\cap B_{R_\alpha-1}(0)$ evolves by rescaled mean curvature flow on $[0,T_\alpha].$ Moreover,
\[
\partial_\eta \left(u_{\alpha(\tau)} + N(u_{\alpha(\tau)}) + v_\alpha(\tau) \right) =0 \qquad \text{on } \partial\mathscr C_+,
\]
so the resulting flow satisfies the free-boundary condition. We denote this flow by $\mathscr T^{\alpha,\mathscr C_+}.$ We also denote by $\mathcal W^{\alpha,\mathscr C_+}$ the graph of $u_\alpha+N(u_\alpha)$ over $\mathscr C_+\cap B_{R_\alpha-1}(0).$

The following lemma follows exactly as in Lemma~4.8 of \cite{ghosh2025cylindrical} (see also Proposition~4.8 of \cite{zhu2020ojasiewicz}). It provides a lower bound for the functional $\phi$ along $\mathscr{T}^{\alpha,\mathscr{C}_+}$

\begin{lemma} {\label{lemma 2.7}}
    There exists a constant $\kappa_0 >0,$ independent of $\alpha,$ such that for all sufficiently small $|\alpha|$ and all $\tau \in [0, T_\alpha]$, we have
    $$\|\phi \big(\mathscr{T}^{\alpha,\mathscr{C}_+}_\tau \big)\|_{L^2} \geq \kappa_0 |\alpha|^2.$$
\end{lemma}

\section{Non-concentration}\label{section 3}

In this section, we prove the non-concentration result in Proposition~\ref{prop 3.2}, similar to Li-Székelyhidi [\citealp{li2024singularity}, Lemma ~34]. We say that a hypersurface $M$ is a $c$-graph over a hypersurface $N$ on the ball $B_R(0)$ if we can write $M \cap B_R(0)$ as the normal graph over $N \cap B_R(0)$ of a function $u$ satisfying $|u|, |\nabla u|, |\nabla^2 u| \leq c.$  

Let $r_0$ be the radius of the Fermi coordinate neighborhood introduced in Section~\ref{section 2.1}. Fix $\eta>0$ sufficiently small, with its precise choice to be determined in Proposition~\ref{prop 3.2}. Here and below, by $M_\tau$ we mean the pullback of the rescaled hypersurface under the Fermi coordinates. In particular, $M_{\tau_0}$ is defined on $B_{r_0 e^{\frac{\tau_0}{2}}}(0)$. 

Let $R$ be the largest radius satisfying
\[
R\leq \min\left\{R_\alpha-1,\; r_0e^{\frac{e^{\tau_0}}{2}}\right\}
\]
such that $M_{\tau_0}$ is an $\eta$-graph over $\mathscr{T}^{\alpha,\mathscr{C}_+}_\tau$ in $B_R(0)$. We define
\[
D_{\mathscr{T}^{\alpha,\mathscr{C}_+}_\tau}(M_{\tau_0}) = \left(\int_{M_{\tau_0}\cap B_R(0)} d_{\mathscr{T}^{\alpha,\mathscr{C}_+}_\tau}^2 e^{-\frac{|x|^2}{4}} \,d\mathcal H^n \right)^{\frac12} + e^{-\frac{R^2}{8p_0}}.
\]
Here $d_{\mathscr{T}^{\alpha,\mathscr{C}_+}_\tau}$ denotes the distance to $\mathscr{T}^{\alpha,\mathscr{C}_+}_\tau$, while $p_0>1$ is a parameter chosen sufficiently close to $1$; its choice will be specified in Proposition~\ref{prop 3.2}. For $\epsilon>0$, it will be convenient to write $R_\epsilon =(8p_0|\ln\epsilon|)^{\frac12}.$

We also define the corresponding Gaussian weighted distance. Let $R$ be the largest radius satisfying 
\[
R\leq \min\left\{R_\alpha-1,\; r_0e^{\frac{e^{\tau_0}}{2}}\right\}
\]
for which $M_{\tau_0}$ can be represented as an $\frac{\eta}{2}$-graph over $\mathcal{W}^{\alpha,\mathscr C_+}$ in $B_R(0)$. We then set
\[
D_{\mathcal{W}^{\alpha,\mathscr C_+}}(M_{\tau_0}) = \left(\int_{M_{\tau_0} \cap B_R(0)} d_{\mathcal{W}^{\alpha,\mathscr C_+}}^2 e^{-\frac{|x|^2}{4}} \,d\mathcal H^n \right)^{\frac12} + e^{-\frac{R^2}{8p_0}}.
\]
When $M$ is a shrinker in $\mathbb{R}^{n+1}_+$ with free boundary on $\{x_{n+1}=0\}$, the additional restriction $R\leq r_0e^{\frac{e^{\tau_0}}{2}}$ is unnecessary. In this case, we take $R$ to be the largest radius satisfying $R\leq R_\alpha-1$ for which $M_{\tau_0}$ can be represented as an $\frac{\eta}{2}$-graph over $\mathcal{W}^{\alpha,\mathscr C_+}$ in $B_R(0)$.

The following is a consequence of the pseudolocality estimate discussed in \ref{section 2.5}. 

\begin{lemma}\label{lemma 3.1}
Let $\eta_0>0$. Then there exist constants $\eta_1>0$ and $R_1>0,$ depending only on $\eta_0$ with the following property. Suppose that $\tau+1\le T_\alpha$ and $M_{\tau_0}$ is an $\eta_1$-graph over
$\mathscr{T}^{\alpha,\mathscr C_+}_{\tau}$ on $B_R(0)$ for some $R>0$. Then for every $s\in[0,1]$, $M_{\tau_0+s}$ is an $\eta_0$-graph over $\mathscr{T}^{\alpha,\mathscr C_+}_{\tau+s}$ on $B_{\min\{e^{\frac{s}{2}}(R-R_1), R_\alpha-1\}}(0).$
\end{lemma}

We will need the following non-concentration estimate. The basic argument can be found in \cite{lotay2022neck}, based on Ecker’s log-Sobolev inequality \cite{ecker2000logarithmic}; similar estimates also appear in \cite{li2024singularity,ghosh2025cylindrical,szekelyhidi2026mean}.

\begin{prop}\label{prop 3.2}
There exist constants $C,\epsilon_0,\theta, T > 0$ and $p>1$ with the following property. Set $\epsilon=D_{\mathscr T^{\alpha,\mathscr C_+}_\tau}(M_{\tau_0}).$ Assume that $|\alpha|,\epsilon,T^{-1}\leq\epsilon_0, \tau+1\leq T_\alpha$, and $\tau_0\geq T$.

If $\epsilon>e^{-\frac{\tau_0}{8}}$, then
$$\left(\int_{M_{\tau_0+1}\cap B_{(1+\theta)R_\epsilon}(0)} d_{\mathscr T^{\alpha,\mathscr C_+}_{\tau+1}}^{2p}
e^{-\frac{|x|^2}{4}}\,d\mathcal H^n \right)^{\frac{1}{2p}} \leq C\epsilon,$$
and $M_{\tau_0+1}$ is an $\frac{\eta}{4}$-graph over $\mathscr T^{\alpha,\mathscr C_+}_\tau$ on $B_{(1+\theta)R_\epsilon}(0)$.

If $\epsilon\leq e^{-\frac{\tau_0}{8}}$, then
$$\left(\int_{M_{\tau_0+1}\cap B_{(1+\theta) \sqrt{\tau_0}}(0)} d_{\mathscr T^{\alpha,\mathscr C_+}_{\tau+1}}^{2p}
e^{-\frac{|x|^2}{4}}\,d\mathcal H^n \right)^{\frac{1}{2p}}\leq Ce^{-\frac{\tau_0}{8}},$$
and $M_{\tau_0+1}$ is an $\frac{\eta}{4}$-graph over $\mathscr T^{\alpha,\mathscr C_+}_\tau$ on $B_{(1+\theta) \sqrt{\tau_0}}(0)$.
\end{prop}

\begin{proof}
Since $D_{\mathscr {T}^{\alpha, \mathscr C_+}_\tau}(M_{\tau_0}) = \epsilon,$ it follows directly from the definition of $D_{\mathscr {T}^{\alpha, \mathscr C_+}_\tau}(M_{\tau_0})$ that $M_{\tau_0}$ is an $\eta$-graph over $\mathscr {T}^{\alpha, \mathscr C_+}_{\tau}$ on $B_{R_\epsilon}(0)$.
    
Fix a sufficiently small constant $\eta' >0$, and let $\eta = \eta_1(\eta')$ to be the constant given by Lemma ~\ref{lemma 3.1}. Since $M_{\tau_0}$ is an $\eta$-graph over $\mathscr {T}^{\alpha, \mathscr C_+}_\tau$ on $B_{R_\epsilon}(0)$, Lemma ~\ref{lemma 3.1} implies that for every $s \in [0,1]$ $M_{\tau_0 + s}$ can be represented as an $\eta'$ graph of a function $v(x,s)$ over $\mathscr{T}^{\alpha, \mathscr C_+}_{\tau +s}$ on $B_{e^{\frac{s}{2}} (R_\epsilon - R_1)}(0)$. Moreover, $|v(\cdot,s)|, |\nabla v(\cdot,s)|, |\nabla^2 v(\cdot,s)| \leq \eta'.$ Choosing $\eta'$ sufficiently small, it follows from Section \ref{section 2.2} that
$$|\partial_s v - \Delta v + \frac{1}{2} x. \nabla v| \leq C \big(|v| + |\nabla v| + R_\epsilon^2 e^{-\frac{\tau_0}{2}}\big).$$
where $C$ is independent of $s$. 

First assume $\epsilon > e^{-\frac{\tau_0}{8}}.$ Then we have $R_\epsilon < \sqrt{p_0\tau_0}.$ Since $\tau_0 \ge T \ge \epsilon_0^{-1}$, by choosing $\epsilon_0$ sufficiently small, we may assume that $\tau_0$ is sufficiently large. Hence
\[
R_\epsilon^2 e^{-\frac{\tau_0}{2}} < p_0\tau_0 e^{-\frac{\tau_0}{2}} \leq e^{-\frac{\tau_0}{4}}.
\]
Arguing as in the proof of [\citealp{li2024singularity}, Lemma $34$], and using Young's inequality, we obtain
$$\partial_s |v|^{\frac{3}{2}} - \Delta |v|^{\frac{3}{2}} + \frac{1}{2} x \cdot \nabla |v|^{\frac{3}{2}} \leq C \Big(|v|^{\frac{3}{2}} + e^{-\frac{\tau_0}{4}} \Big),$$
after possibly increasing $C$. 
        
As in \cite[Proposition~5.2]{ghosh2025cylindrical}, the weighted $L^p$-norms of $d^2_{\mathscr {T}^{\alpha, \mathscr C_+}_{\tau+s}}$ on $M_{\tau_0+s}$ and $|v|^2$ on $\mathscr {T}^{\alpha, \mathscr C_+}_{\tau+s}$ are uniformly equivalent. Define  
$$f(x,s) = e^{-Cs} |v|^{\frac{3}{2}} - Cse^{-\frac{\tau_0}{4}}.$$
Then $f$ is a subsolution of the drift heat equation. Also, $e^{-s} |x|^2 - (R_\epsilon - R_1 - 1)^2$ is a subsolution of the drift heat equation. We therefore define
$$\widehat{f} =
\begin{cases}
    \max \big\{f, e^{-s} |x|^2 - (R_\epsilon - R_1 - 1)^2\big\}, & \text{for}\; |x| \leq  e^{\frac{s}{2}} (R_\epsilon - R_1), \\
    e^{-s}|x|^2 - (R_\epsilon - R_1 - 1)^2, & \text{for}\; |x| > e^{\frac{s}{2}} (R_\epsilon - R_1).
\end{cases}$$
Note that, for $|x| \leq e^{\frac{s}{2}} (R_\epsilon - R_1 - 1)$, we have $\widehat f = f,$ while for $|x|$ close to $e^{\frac{s}{2}} (R_\epsilon - R_1),$ we have $\widehat f = e^{-s} |x|^2 - (R_\epsilon - R_1 - 1)^2$ using $|v| \leq \eta.$ It follows that $\widehat{f}$ is also a subsolution of the drift heat equation along $\mathscr {T}^{\alpha, \mathscr C_+}_{\tau+s}$ on the time interval $s \in [0, 1].$ Now,
$$\widehat{f}^\frac{4}{3} \leq 
\begin{cases}
    |v|^2, & \text{for}\; |x| \leq e^{\frac{s}{2}} (R_\epsilon - R_1 - 1), \\
    (|v|^2 + |x|^{\frac{8}{3}} ), & \text{for}\; e^{\frac{s}{2}} (R_\epsilon - R_1 - 1) < |x| \leq e^{\frac{s}{2}} (R_\epsilon - R_1),\\
    |x|^{\frac{8}{3}}, & \text{for}\; |x| > e^{\frac{s}{2}} (R_\epsilon - R_1).
\end{cases}$$
Choose $q > 1$ such that $1< q < p_0.$ Then for sufficiently small $\epsilon > 0$, 
$$\int_{M_{\tau_0} \setminus B_{R_\epsilon - R_1 - 1}(0)} |x|^{\frac{8}{3}} e^{-\frac{|x|^2}{4}} d\mathcal{H}^n \leq e^{-\frac{(R_\epsilon - R_1 - 1)^2}{4q}} \leq e^{-\frac{R_\epsilon^2}{4p_0}} \leq C\epsilon^2,$$
Consequently,
$$\int_{\mathscr {T}^{\alpha, \mathscr C_+}_{\tau}} \widehat{f}^{\frac{4}{3}} e^{-\frac{|x|^2}{4}} d\mathcal{H}^n \leq C\epsilon^2.$$
Applying Ecker’s log-Sobolev inequality \cite{ecker2000logarithmic}, we obtain a $p > 1$ and a $C >0$, such that
$$\int_{\mathscr {T}^{\alpha, \mathscr C_+}_{\tau+1}} \widehat f^{\frac{4p}{3}} \;e^{-\frac{|x|^2}{4}} d \mathcal{H}^n \leq C \epsilon^{2p}.$$
Choose $\theta > 0$  sufficiently small so that $1+ \theta < \min \{e^{\frac{1}{2}} , \sqrt p \}.$ Then for all sufficiently small $\epsilon > 0$,  
$$e^{\frac{1}{2}}(R_\epsilon - R_1) \geq (1+ \theta) R_\epsilon.$$
Moreover, on $B_{e^{\frac{1}{2}} (R_\epsilon - R_1)}(0)$, we have
$$f(.,1) = e^{-C} |v(x,1)|^{\frac{3}{2}} - Ce^{-\frac{\tau_0}{4}} \leq \widehat f(.,1).$$
Since $R_\epsilon \leq \sqrt{p_0\tau_0},$ we have $(1+\theta)R_\epsilon \leq C\sqrt{\tau_0}.$ Therefore, by the area ratio bound, for $\tau_0$ sufficiently large,
$$\int_{M_{\tau_0 + 1} \cap B_{(1+ \theta) R_\epsilon}(0)} e^{-\frac{p\tau_0}{3}} e^{-\frac{|x|^2}{4}} d \mathcal{H}^n \leq C e^{-\frac{p\tau_0}{3}} \leq \epsilon^{2p}.$$
It follows that
$$\int_{M_{\tau_0 + 1} \cap B_{(1+ \theta) R_\epsilon}(0)} d_{\mathscr {T}^{\alpha, \mathscr C_+}_{\tau +1}}^{2p} \;e^{-\frac{|x|^2}{4}} d \mathcal{H}^n \leq C \epsilon^{2p}.$$
The graphical conclusion follows from the same argument as in \cite[Proposition~5.2]{ghosh2025cylindrical}.

If $\epsilon\leq e^{-\frac{\tau_0}{8}}$, then $R_\epsilon \geq \sqrt{p_0\tau_0}$, and hence, for $\tau_0$ sufficiently large, the above argument can be carried out with $\sqrt{p_0\tau_0}$ in place of $R_\epsilon$, yielding the desired conclusion.
\end{proof}

The non-concentration estimate will be used in the following form. Its proof is identical to that of Proposition~5.3 in \cite{ghosh2025cylindrical}, and we therefore omit it.

\begin{cor} \label{cor 3.3}
There exist $\epsilon_0, C_0 > 0$ with the following property. For every 
$\delta>0$, there exists $R_0 = R_0(\delta) >0$ such that, if $\tau + 1 \leq T_\alpha$, $|\alpha| \leq \epsilon_0, D_{\mathscr {T}^{\alpha, \mathscr C_+}_\tau}(M_{\tau_0}) \leq \epsilon_0$ and $D_{\mathscr {T}^{\alpha, \mathscr C_+}_\tau}(M_{\tau_0}) > e^{-\frac{\tau_0}{8}}$ then the following improved estimate holds:
\[
D_{\mathscr {T}^{\alpha, \mathscr C_+}_{\tau+1}}(M_{\tau_0+1}) \leq \max \bigg\{\delta D_{\mathscr {T}^{\alpha, \mathscr C_+}_\tau}(M_{\tau_0}) + C_0 \bigg(\int_{M_{\tau_0+1}\cap B_{R_0}(0)} d_{\mathscr {T}^{\alpha, \mathscr C_+}_{\tau+1}}^2 e^{-\frac{|x|^2}{4}} d\mathcal{H}^n \bigg)^{\frac12}, C_0e^{-\frac{(R_\alpha-1)^2}{8p_0}} \bigg\}.
\]
\end{cor}

Let $u$ be a solution to the linearized equation $(\partial_s - L)u = 0$ on $\mathscr{C}_+$ subject to the Neumann boundary condition $\partial_\eta u=0$ on $\partial\mathscr C_+.$ Define
$$I^2(u,s) = \int_\mathscr{C_+} u^2 e^{-\frac{|x|^2}{4}} d\mathcal{H}^n.$$
The following lemma follows by the same argument as in
\cite{ghosh2025cylindrical}.

\begin{lemma}\label{lemma 3.4} 
Let $u$ be a solution to the linearized equation. There exist constants $0<\delta_1<\delta_2<1$ such that the following statements hold. \begin{enumerate} 
\item If $I(u,s+1)\geq e^{\delta_1}I(u,s),$ then $I(u,s+2)\geq e^{\delta_2}I(u,s+1).$ 
\item If, at some time $s$, the decomposition of $u$ into eigenfunctions of $L$ has no component in the zero eigenspace, then either $I(u,s+1)\geq e^{\delta_2}I(u,s)$ or $I(u,s-1)\geq e^{\delta_2}I(u,s).$ \end{enumerate} 
\end{lemma}

We introduce the following modified version of the distance function:
$$\mathcal {D}_{\mathscr {T}^{\alpha, \mathscr C_+}_\tau}(M_{\tau_0}) := \sup_{s \in [0,2]} D_{\mathscr {T}^{\alpha, \mathscr C_+}_{\tau-s}}(M_{\tau_0-s}).$$
Note that $\mathcal{D}_{\mathscr {T}^{\alpha, \mathscr C_+}_\tau}(M_{\tau_0})$ provides a uniform bound for $M_{\tau_0-s}$ for $s \in [0,2].$

Combining Proposition~\ref{prop 3.2} and Lemma~\ref{lemma 3.4} with the non-concentration estimate in Corollary~\ref{cor 3.3}, we obtain the following three-annulus lemma for the distance function. Its proof follows by adapting the argument of \cite[Lemma~5.6]{ghosh2025cylindrical} (see also \cite[Lemma~36]{li2024singularity}). The only additional point in the free-boundary setting is to verify the Neumann boundary condition for the limiting function. The boundary error estimate derived in Section~\ref{section 2.3}, together with the pointwise estimate in Proposition~\ref{prop 3.2}, implies that, under the normalization used in the contradiction argument, the boundary error is of higher order and converges to zero. Hence the limiting function satisfies the Neumann boundary condition on $\partial\mathscr C_+.$

\begin{lemma} \label{lemma 5.6}
Let $\delta_1, \delta_2$ be the constants from Lemma ~\ref{lemma 3.4} and $C_0$ be the constant from Corollary \ref{cor 3.3}. Then there exist $\epsilon_0 > 0$ and a large $L >0$ such that the following holds. Suppose that 
\begin{enumerate}
    \item $\tau + 2L \le T_\alpha,$
    \item $\tilde{\delta_1},\tilde{\delta_2} \in (\delta_1, \delta_2)$ with $\tilde{\delta_1} < \tilde{\delta_2},$
    \item $|\alpha| \leq \epsilon_0,$
    \item $\mathcal{D}_{\mathscr {T}^{\alpha, \mathscr C_+}_\tau}(M_{\tau_0}) \leq \epsilon_0.$
    \item $\mathcal{D}_{\mathscr {T}^{\alpha, \mathscr C_+}_{\tau+L}}(M_{\tau_0+L}) > \max\Big\{e^{-\frac {\tau_0} {8}}, C_0  e^{- \frac{(R_\alpha-1)^2}{8p_0}}\Big\}.$
\end{enumerate}
If 
$$\mathcal{D}_{\mathscr {T}^{\alpha, \mathscr C_+}_{\tau+L}}(M_{\tau_0 + L}) \geq e^{\tilde{\delta_1}L} \mathcal{D}_{\mathscr {T}^{\alpha, \mathscr C_+}_\tau} (M_{\tau_0}),$$
then
$$\mathcal{D}_{\mathscr {T}^{\alpha, \mathscr C_+}_{\tau + 2L}}(M_{\tau_0 + 2L}) \geq e^{\tilde{\delta_2}L} \mathcal{D}_{\mathscr {T}^{\alpha, \mathscr C_+}_{\tau+L}}(M_{\tau_0 + L}).$$
\end{lemma}

\section{The main argument}\label{section 6}

Let $\mathcal{M}_+$ denote the space of Radon measures on $\mathbb{R}^{n+1}$ with finite Gaussian mass and whose support is contained in the $\mathbb{R}^{n+1}_+.$ Following \cite{colding2015rigidity}, we define a metric \(d_V\) on $\mathcal{M}_+$. Let \(\{f_j\}_{j=1}^{\infty}\) be a countable dense subset of the unit ball of
\(C^0_c(\mathbb{R}^{n+1})\), equipped with the $C^0$ norm. For $\mu_1, \mu_2 \in \mathcal{M}_+$  define
\begin{equation*}
d_V(\mu_1,\mu_2) = \sum_{j=1}^{\infty}2^{-j} \left|\int_{\mathbb{R}^{n+1}} f_j e^{-\frac{|x|^2}{4}}\,d\mu_1 -
\int_{\mathbb{R}^{n+1}} f_j e^{-\frac{|x|^2}{4}}\,d\mu_2 \right|.
\end{equation*}
Since all measures under consideration are supported in $\mathbb{R}^{n+1}_+,$ the above metric induces the weak topology on $\mathcal{M}_+$. Let
$$\mathcal{W} = \Big\{\mathcal{W}^{\alpha, \mathscr{RC}_+}: |\alpha| \leq \epsilon_1, \;\mathscr{R} \in K_1^+\Big\},$$
where $\mathscr{RC}_+$ denotes the image of $\mathscr{C}_+$ under rotation $\mathscr{R}.$ Here $\epsilon_1 > 0$ is a fixed constant so that $\mathcal{W}^{\alpha, \mathscr{RC}_+}, \mathscr{T}^{\alpha,\mathscr{RC}_+}$ are defined for $|\alpha| \leq \epsilon_1$ and Lemma~\ref{lemma 2.7} holds. We define the distance between two hypersurfaces in $\mathbb{R}^{n+1}_+$ as follows,
$$D(M_1, M_2) = \inf_{W} \Big\{D_W(M_1) + D_W(M_2) :  W \in \mathcal{W} \Big\}.$$
Recall 
\[
\rho(x)=(4\pi)^{-\frac n2}e^{-\frac{|x|^2}{4}}, \qquad
\phi(x,\tau) = \eta\!\left(\kappa^{-\frac12}e^{-\frac{\tau}{4}}\bigl(|x|^2-\alpha\bigr)\right), \qquad
\eta(s)=(1-s)_+^4,
\]
and
\[
\Phi(x,\tau)=\rho(x)\phi(x,\tau),
\]
where $A$ and $\mathcal{E}_0$ are from Theorem \ref{thm 2.3}. We define
$$\mathcal{A}(M_\tau) = e^{A e^{-\frac{\tau}{4}}}\int_{M_\tau} \Phi(x,\tau) d\mathcal H^n_{g_\tau} + A\mathcal E_0e^{-\tau} - \int_{\mathscr{C}_+} \rho(x) d\mathcal H^n.$$
For $\gamma > 0$, we define
\[
\mathcal A^\gamma(M) := |\mathcal A(M)|^{\gamma-1}\mathcal A(M).
\]
The following lemmas follow from \cite{ghosh2025cylindrical}.
\begin{lemma}\label{lemma 4.1}
Let $N\subset \mathbb{R}^{n+1}$ be a hypersurface satisfying $|\langle x,n_N(x)\rangle|\leq C_0$. Suppose that $M$ be a normal graph over $N$ on $B_R(0)$ given by a function $u$ with sufficiently small $C^2$ norm. Then, for every $f\in C^1(\mathbb{R}^{n+1})$, there exists a constant $C$ such that
\begin{multline*}
\left|\int_{M\cap B_R(0)} f e^{-\frac{|x|^2}{4}}\,d\mathcal H^n - \int_{N\cap B_R(0)} f e^{-\frac{|x|^2}{4}}\,d\mathcal H^n \right| \\
\leq C\|\nabla f\|_{C^0}\|u\|_{L^2} + C\|f\|_{C^0} \left(\|u\|_{L^2}\|\phi_N\|_{L^2} +\|u\|_{L^2}\|u\|_{H^2} + e^{-\frac{(R-1)^2}{4q}} \right),
\end{multline*}
where all $L^2$ and $H^2$ norms are taken over $N\cap B_R(0)$.
\end{lemma}

\begin{lemma}\label{lemma 4.2}
For any two hypersurfaces $M_1$ and $M_2$, we have
\[
d_V(\mu_{M_1},\mu_{M_2}) \leq CD(M_1,M_2),
\]
where $\mu_{M_i}=\mathcal H^n\lfloor M_i$ denotes the Radon measure associated to $M_i$.
\end{lemma}

\begin{prop}\label{prop 4.3}
There exists $\epsilon>0$ such that if $|\alpha| \leq \epsilon$, then for every $\tau\in[0,T_\alpha]$,
\[
\left|\int_{\mathscr{T}^{\alpha,\mathscr C_+}_{\tau}} \rho \; d\mathcal H^n - \int_{\mathscr{C}_+} \rho \; d\mathcal H^n \right| \leq |\alpha|^{\frac52}.
\]
\end{prop}

\begin{prop} \label{prop 4.4}
There are $C, \kappa_1, \epsilon_0 > 0$ and $\gamma \in (0,1)$ with the following property. Suppose that $|\alpha| \leq \epsilon_0, \tau + 1 \leq T_\alpha$  and $e^{-\frac{\tau_0}{8}}, \mathcal{D}_{\mathscr {T}^{\alpha, \mathscr{C}_+}_\tau} (M_{\tau_0}) \leq \kappa_1 |\alpha|^2.$ Then we have 
$$\mathcal{A}^\gamma(M_{\tau_0}) - \mathcal{A}^\gamma(M_{\tau_0 + 1}) \geq C^{-1} |\alpha|^2.$$
\end{prop}

\begin{proof}
Let $\epsilon_0 >0$ be the constant from Proposition~\ref{prop 3.2} and set $ \epsilon = \mathcal {D}_{\mathscr {T}^{\alpha, \mathscr{C}_+}_\tau} (M_{\tau_0})$. First suppose that $e^{-\frac{\tau_0}{8}} < \epsilon \leq \kappa_1|\alpha|^2.$ By Proposition~\ref{prop 3.2}, for $s \in [0,1], M_{\tau_0+s}$ can be written as a graph over $\mathscr {T}^{\alpha, \mathscr{C}_+}_{\tau+s}$ on $B_{R_\epsilon}(0)$, with graph function $u$ satisfying $|u|, |\nabla u|, |\nabla^2 u| \leq Ce^{\frac{|x|^2}{4p}} \epsilon$. Using $R_\epsilon \leq \sqrt{p_0\tau_0}, e^{-\frac{\tau_0}{8}} < \epsilon \leq \kappa_1|\alpha|^2$, the triangle inequality, and Lemma~\ref{lemma 4.1}, we obtain
\begin{align*}
\bigg|&\int_{M_{\tau_0+s} \cap B_{R_\epsilon}(0)} \Phi(x,\tau_0+s) \,d\mathcal{H}^n_{g_{\tau_0 +s}} - \int_{\mathscr{T}^{\alpha,\mathscr C}_{\tau+s} \cap B_{R_\epsilon}(0)} \Phi(x,\tau_0+s) \,d\mathcal{H}^n \bigg| \\
&\leq C R_\epsilon e^{-\frac{\tau_0}{2}} \int_{M_{\tau_0+s} \cap B_{R_\epsilon}(0)} \Phi(x,\tau_0+s) \,d\mathcal{H}^n \\
&\hspace{0.4 cm} + C R_\epsilon  e^{-\frac{\tau_0}{4}} \|u_{\tau_0+s}\|_{L^2}  + C\left(\|u_{\tau_0+s} \|_{L^2}\|\phi_{\mathscr{T}^{\alpha,\mathscr C}_{\tau+s}}\|_{L^2} +\|u_{\tau_0+s}\|_{L^2}\|u_{\tau_0+s}\|_{H^2} + e^{-\frac{(R_\epsilon-1)^2}{4q}} \right) \\
&\leq C \kappa_1 |\alpha|^4.
\end{align*}
Also,
\begin{align*}
\bigg|&\int_{M_{\tau_0+s}\setminus B_{R_\epsilon}(0)} \Phi(x,\tau_0+s) \,d\mathcal{H}^n_{g_{\tau_0 +s}} - \int_{\mathscr{T}^{\alpha,\mathscr C}_{\tau+s}\setminus B_{R_\epsilon}(0)} \Phi(x,\tau_0+s) \,d\mathcal{H}^n \bigg| \\
&\leq C \int_{M_{\tau_0+s}\setminus B_{R_\epsilon}(0)} \Phi(x,\tau_0+s) d\mathcal{H}^n + \int_{\mathscr{T}^{\alpha,\mathscr C}_{\tau+s}\setminus B_{R_\epsilon}(0)} \Phi(x,\tau_0+s) d\mathcal{H}^n \\
&\leq Ce^{-\frac{R_\epsilon^2}{4p_0}} \leq C \kappa_1|\alpha|^4.
\end{align*}
By Theorem \ref{thm 2.3}, we have
$$\int_{\mathscr{T}^{\alpha,\mathscr C_+}_{\tau}} \Phi(x,\tau_0) d\mathcal H^n - \int_{\mathscr{T}^{\alpha,\mathscr C_+}_{\tau+1}} \Phi(x,\tau_0+1) d\mathcal H^n = \int_0^{1} \int_{\mathscr {T}^{\alpha, \mathscr{C}_+}_{\tau+s}} \phi^2 \Phi(x,\tau_0+s) d\mathcal{H}^n.$$
Note that when $|x|^2 \leq \alpha+\left(1-2^{-1/4}\right)\kappa^{1/2}e^{\tau_0/4}, \phi(x,\tau_0+s) \geq \frac{1}{2}.$ Therefore, by Lemma~\ref{lemma 2.7}, we get 
$$\int_{\mathscr{T}^{\alpha,\mathscr C_+}_{\tau}} \Phi(x,\tau_0) d\mathcal H^n - \int_{\mathscr{T}^{\alpha,\mathscr C_+}_{\tau+1}} \Phi(x,\tau_0+1) d\mathcal H^n \geq \frac{1}{4} \kappa_0^2|\alpha|^2.$$
By area ratio bound,
$$e^{A e^{-\frac{\tau_0 + s}{4}}} \int_{M_{\tau_0 + s}} \Phi(x,\tau) d\mathcal H^n_{g_{\tau_0 + s}} - \int_{M_{\tau_0 + s}}  \Phi(x,\tau) d\mathcal H^n_{g_{\tau_0 + s}} \leq Ce^{-\frac{\tau_0}{4}}.$$
Therefore, 
\begin{align*}
    \mathcal{A}(M_{\tau_0}) - \mathcal{A}(M_{\tau_0+1}) &\geq e^{A e^{-\frac{\tau_0}{4}}}\int_{M_{\tau_0}} \Phi(x,\tau_0) d\mathcal H^n_{g_{\tau_0}} - e^{A e^{-\frac{\tau_0 + 1}{4}}}\int_{M_{\tau_0+1}} \Phi(x,\tau_0+1) d\mathcal H^n_{g_{\tau_0+1}} \\
    &\geq \int_{\mathscr{T}^{\alpha,\mathscr C_+}_{\tau}} \Phi(x,\tau_0) d\mathcal H^n - \int_{\mathscr{T}^{\alpha,\mathscr C_+}_{\tau+1}} \Phi(x,\tau_0+1) d\mathcal H^n - Ce^{-\frac{\tau_0}{4}}  \\
    &\quad - \bigg |\int_{M_{\tau_0}} \Phi(x,\tau_0) d\mathcal H^n_{g_{\tau_0}} - \int_{\mathscr{T}^{\alpha,\mathscr C_+}_{\tau}} \Phi(x,\tau_0) d\mathcal H^n \bigg| \\ 
    &\quad - \bigg |\int_{M_{\tau_0 +1}} \Phi(x,\tau_0+1) d\mathcal H^n_{g_{\tau_0+1}} - \int_{\mathscr{T}^{\alpha,\mathscr C_+}_{\tau+1}} \Phi(x,\tau_0+1) d\mathcal H^n \bigg|\\
    &\geq \frac{1}{4}\kappa_0^2 |\alpha|^4 - 3C \kappa_1 |\alpha|^4 \\
    &\geq \frac{\kappa_0^2}{8} |\alpha|^4.
\end{align*}
where we have chosen $\kappa_1$ small such that $\kappa_1 < \frac{\kappa_0^2}{24C}.$ Moreover, for $s\in[0,1]$, we have
\[
\bigg|\int_{M_{\tau_0+s}} \Phi(x, \tau_0+s) \,d\mathcal{H}^n_{g_{\tau_0 +s}} - \int_{\mathscr{T}^{\alpha,\mathscr C}_{\tau+s}} \Phi(x, \tau_0+s) \,d\mathcal{H}^n \bigg| \leq C|\alpha|^4.
\]
We also have,
$$\bigg|\int_{\mathscr{T}^{\alpha,\mathscr C}_{\tau+s}} \Phi(x, \tau_0+s) \,d\mathcal{H}^n - \int_{\mathscr{T}^{\alpha,\mathscr C}_{\tau+s}} \rho\,d\mathcal{H}^n \bigg| \leq C e^{-\frac{\tau_0}{4}},$$
and by Proposition~\ref{prop 4.3},
\[
\left|\int_{\mathscr{T}^{\alpha,\mathscr C_+}_{\tau}} \rho d\mathcal H^n - \int_{\mathscr{C}_+} \rho d\mathcal H^n \right| \leq |\alpha|^{\frac52}.
\]
Therefore, by the triangle inequality, for sufficiently small $|\alpha|, |\mathcal{A}(M_{\tau_0+s})| \leq 2|\alpha|^{\frac52}.$ Now, 
\begin{align*}
\mathcal{A}(M_{\tau_0})^\gamma -\mathcal{A}(M_{\tau_0+1})^\gamma 
&= \gamma\int_0^1 |\mathcal{A}(M_{\tau_0+s})|^{\gamma-1} \left(-\frac{d}{ds}\mathcal{A}(M_{\tau_0+s})\right)\,ds\\
&\geq \gamma 2^{\gamma-1} |\alpha|^{\frac{5\gamma}{2}-\frac52} \left(\mathcal{A}(M_{\tau_0}) -\mathcal{A}(M_{\tau_0+1}) \right) \\
&\geq \gamma \kappa_0^2 2^{\gamma -4}|\alpha|^{\frac{5\gamma}{2} - \frac12} |\alpha|^2.
\end{align*}
Choosing $0<\gamma<\frac15,$ the required result follows.

Finally, suppose that $\epsilon \leq e^{-\frac{\tau_0}{8}}\leq \kappa_1|\alpha|^2.$ We then use the second conclusion of Proposition~\ref{prop 3.2}, with \(\sqrt{p_0\tau_0}\) in place of \(R_\epsilon\). The remainder of the argument is identical to the case $e^{-\frac{\tau_0}{8}} < \epsilon \leq \kappa_1|\alpha|^2.$

\end{proof}
We have the following Lemma from \cite{ghosh2025cylindrical}.

\begin{lemma}\label{lemma 4.5}
There exist constants $C>0$ and $\epsilon>0$ such that the following holds. Suppose that $|\alpha|, |\widehat{\alpha}|, |\mathscr{R}-Id|, D_{\mathscr{T}^{\alpha,\mathscr{C}_+}_\tau}(M_{\tau_0}), D_{\mathcal{W}^{\alpha(\tau),\mathscr{C}_+}}(M_{\tau_0}), D_{\mathscr{T}^{\alpha,\mathscr{C}_+}_{\tau+L}}(M_{\tau_0+L})\leq\epsilon,$ and $ D_{\mathscr{T}^{\alpha,\mathscr{C}_+}_\tau}(M_{\tau_0})>e^{-\tau_0/8}.$
Then
\[
D_{\mathscr{T}^{\widehat{\alpha},\mathscr{RC}_+}_{\tau+1}}(M_{\tau_0+1}),\ 
D_{\mathcal{W}^{\widehat{\alpha}(\tau+1),\mathscr{RC}_+}}(M_{\tau_0+1}) \leq C\bigg(D_{\mathscr{T}^{\alpha,\mathscr{C}_+}_\tau}(M_{\tau_0}) +|\mathscr{R}-Id| +|\alpha-\widehat{\alpha}|
+e^{-\frac{(R_{\widehat{\alpha}}-1)^2}{8p_0}}\bigg),
\]
\vspace{-\baselineskip}
\[
D_{\mathscr{T}^{\widehat\alpha,\mathscr{RC}_+}_\tau}(M_{\tau_0}) \leq C\bigg(D_{\mathcal{W}^{\alpha(\tau),\mathscr{C}_+}}(M_{\tau_0}) +|\mathscr{R}-Id| +|\alpha-\widehat{\alpha}|
+e^{-\frac{(R_{\widehat{\alpha}}-1)^2}{8p_0}}\bigg).
\]
If $|\widehat{\alpha}|\geq|\alpha|$, the exponential error terms are bounded by the corresponding distance terms. Finally,
$$ D_{\mathcal{W}^{\alpha(\tau+1),\mathscr{C}_+}}(M_{\tau_0+1}) + D_{\mathcal{W}^{\alpha(\tau+1),\mathscr{C}_+}}(M_{\tau_0+L+1})
\leq C\Big(D_{\mathscr{T}^{\alpha,\mathscr{C}_+}_\tau}(M_{\tau_0}) +
D_{\mathscr{T}^{\alpha,\mathscr{C}_+}_{\tau+L}}(M_{\tau_0+L}) +|\alpha|^2 \Big).$$
\end{lemma}

The following proposition shows that any nearby shrinker with free boundary on \(\{x_{n+1}=0\}\) is a rotation of \(\mathscr{C}_+\) by some \(\mathscr{R}\in K^1_+\). The proof is identical to that of \cite[Proposition ~6.7]{ghosh2025cylindrical}, and we omit it.

\begin{prop} \label{prop 4.6}
    There exists $\epsilon > 0$ such that the following holds. Suppose that $M$ is a shrinker in $\mathbb{R}^{n+1}_+$ with free boundary on $\{x_{n+1}=0\}$, such that $d_V(\mu_{\mathscr{C}_+}, \mu_M) \leq \epsilon.$ Then $M=\mathscr{R}\mathscr{C}_+$ for some $\mathscr{R}\in K^1_+$.
\end{prop}

As in \cite{ghosh2025cylindrical}, we define
$$\mathcal {D} (M_{\tau_1},M_{\tau_2}) := \sup_{s \in [0,1]} D(M_{\tau_1 -s}, M_{\tau_2 -s}).$$
The following proposition provides the key estimate needed to prove uniqueness of the tangent flow. Its proof follows the argument of Proposition~6.8 in \cite{ghosh2025cylindrical}; we indicate only the modifications needed in the present free-boundary setting.

\begin{prop} \label{prop 4.7}
There are $C, L, \epsilon > 0$ and $\gamma \in (0,1)$ with the following property. Suppose that 
$$\tau_0^{-1}, \sup\limits_{s\in [0,1]} d_V(\mu_{\mathscr{C}_+}, \mu_{M_{\tau_0-s}}), \mathcal A^\gamma (M_{\tau_0 - 1}) - \lim\limits_{\tau \to \infty} \mathcal A^\gamma(M_\tau) \leq \epsilon.$$
Then one of the following holds:
\begin{enumerate} [label=\normalfont(\roman*) ]
\item $\mathcal{D}(M_{\tau_0 + L}, M_{\tau_0 + 2L}) \leq Ce^{-\frac{\tau_0}{8}}.$
\item $\mathcal{D}(M_{\tau_0 + L}, M_{\tau_0 + 2L}) \leq \frac{1}{2} \mathcal{D}(M_{\tau_0}, M_{\tau_0 + L}).$
\item $\mathcal{D}(M_{\tau_0 + L}, M_{\tau_0 + 2L}) \leq C \Big(\mathcal A^\gamma(M_{\tau_0 + L}) -\mathcal A^\gamma(M_{\tau_0 + 2L}) \Big).$
\end{enumerate}
\end{prop}

\begin{proof}
Suppose, by contradiction, that the conclusion fails. Then there exists a sequence of flows $M^i$ such that
$\tau_i^{-1}, \sup\limits_{s\in [0,1]} d_V(\mu_{\mathscr{C}^i_+}, \mu_{M^i_{\tau_i-s}}), \mathcal A^\gamma(M^i_{\tau_i-1}) - \lim\limits_{\tau \to \infty} \mathcal A^\gamma(M^i_\tau) \leq \epsilon_i,$ where $\epsilon_i \to 0,$ but none of the three alternatives holds. 

Up to choosing a subsequence, we can replace the sequence $\mathscr{C}^i_+$ by a single $\mathscr{C}_+.$ Let $\tilde{\delta}_1,\tilde{\delta}_2$ be the constants from Lemma~\ref{lemma 5.6}. As in the proof of \cite[Proposition~6.8]{ghosh2025cylindrical}, after passing to a subsequence, we can find $\alpha^i\to0$ and rotations $\mathscr R^i\to Id$ such that one of the following possibilities holds: 
\begin{enumerate} [label=\normalfont(\alph*)]
    \item $\mathcal{D}_{\mathscr{T}^{\alpha^i, \mathscr{R}^i \mathscr{C}_+}_{1+2L}} (M^i_{\tau_i + 2L}) \geq e^{\tilde{\delta_1} L} \mathcal{D}_{\mathscr{T}^{\alpha^i, \mathscr{R}^i \mathscr{C}_+}_{1+L}} (M^i_{\tau_i + L}),$ 
    \item $\mathcal{D}_{\mathscr{T}^{\alpha^i, \mathscr{R}^i \mathscr{C}_+}_2} (M^i_{\tau_i + 1}) \geq e^{\tilde{\delta_2} L} \mathcal{D}_{\mathscr{T}^{\alpha^i, \mathscr{R}^i \mathscr{C}_+}_{1+L}} (M^i_{\tau_i + L}).$ 
\end{enumerate}
Choose $\mathscr R^i$ and $\alpha^i$ so that
$$d_i:= \mathcal{D}_{\mathscr{T}^{\alpha^i,\mathscr R^i\mathscr C_+}_{1+L}} (M^i_{\tau_i+L}) $$
is within a factor of $2$ of the infimum among all admissible comparison flows. Since $M^i$ converges to $\mathscr C_+$, we have $d_i\to0$.

Suppose first that neither~(a) nor~(b) holds and $d_i>e^{-\frac{\tau_i}{8}}.$ Then after passing to a subsequence, we can argue exactly as in \cite{ghosh2025cylindrical} to write $M^i_{\tau_i+s}$ as the graph of $u_i(s)$ over $\mathscr{T}^{\alpha^i,\mathscr C^i_+}_{1+s}$ on larger and larger subsets for $s\in[-1,2L]$, where $\mathscr C^i_+=\mathscr R^i\mathscr C_+$, and the rescaled functions $d_i^{-1}u_i$ converge locally smoothly to a solution $u$ of the drift heat equation on $\mathscr C_+$ and by Section \ref{section 2.3}, 
$$\partial_\eta u = 0 \qquad \text{on} \;\partial \mathscr{C}_+.$$
The remainder of the contradiction argument is then identical to that of \cite[Proposition~6.8]{ghosh2025cylindrical}, yielding a contradiction for $L$ sufficiently large. Once we establish that either ~(a) or ~(b) holds, we can argue as in \cite{ghosh2025cylindrical} to obtain either (ii) or (iii).
 
It remains to consider the case $d_i\leq e^{-\frac{\tau_i}{8}}.$ If $e^{-\frac{\tau_i}{8}}\geq \kappa |\alpha^i|^2,$ then Proposition ~\ref{prop 3.2} together with Lemma \ref{lemma 4.5} gives

$$\mathcal{D}(M^i_{\tau_i+L},M^i_{\tau_i+2L}) \leq C e^{-\frac{\tau_i}{8}}, $$
after enlarging $C$ and absorbing $\kappa^{-1}$ into the constant. Thus alternative~(i) holds.

It remains to consider $e^{-\frac{\tau_i}{8}}<\kappa |\alpha^i|^2.$ As $d_i \leq e^{-\frac{\tau_i}{8}}$, Proposition~\ref{prop 4.4} gives
$$\mathcal A^\gamma(M^i_{\tau_i+L}) - \mathcal A^\gamma(M^i_{\tau_i+2L}) \geq C^{-1}|\alpha^i|^2.$$
On the other hand, Proposition~\ref{prop 3.2}, together with Lemma~\ref{lemma 4.5}, gives
$$\mathcal{D}(M^i_{\tau_i+L},M^i_{\tau_i+2L}) \leq C|\alpha^i|^2.$$
Combining the last two estimates yields
$$\mathcal{D}(M^i_{\tau_i+L},M^i_{\tau_i+2L}) \leq C\Big(\mathcal A^\gamma(M^i_{\tau_i+L}) - \mathcal A^\gamma(M^i_{\tau_i+2L})\Big),$$
and hence alternative~(iii) holds.
\end{proof}

We can now prove Theorem~\ref{thm 1.1}. The argument is similar to that of \cite[Theorem~6.7]{szekelyhidi2020uniqueness}. Set
$$E_0:=\mathcal A^\gamma(M_{\tau_0-1})-\lim_{\tau\to\infty}\mathcal A^\gamma(M_\tau).$$
Assume that $E_0$ is sufficiently small and that $M_{\tau_0+kL-s}$ remains sufficiently close to $\mathscr C_+$ for $s\in[0,1]$ and $k=0,\ldots,N$. Applying Proposition~\ref{prop 4.7} successively, the intervals for which (ii) holds give a geometric contraction, those for which (iii) holds are controlled by the corresponding drops of $\mathcal A^\gamma$, while (i) contributes a summable error. Thus
$$\sum_{k=0}^{N}\mathcal D(M_{\tau_0+kL},M_{\tau_0+(k+1)L}) \leq C\mathcal D(M_{\tau_0},M_{\tau_0+L}) + CE_0 + Ce^{-\frac{\tau_0}{8}}.$$
Consequently, by Lemma~\ref{lemma 4.2},
$$\sup_{s\in[0,1]}d(M_{\tau_0-s},M_{\tau_0+(N+1)L-s}) \leq C\mathcal D(M_{\tau_0},M_{\tau_0+L}) + CE_0 + Ce^{-\frac{\tau_0}{8}}.$$
For $\tau_0$ sufficiently large and $E_0$ sufficiently small, this shows that the flow remains sufficiently close to $\mathscr C_+$, so the proposition can be iterated indefinitely. Finally, since $\mathscr C_+$ is a tangent flow, there is a sequence $\tau_i\to\infty$ such that
$$\sup_{s\in[0,1]}d(M_{\tau_i-s},\mathscr C_+) \to 0, \qquad \mathcal A^\gamma(M_{\tau_i}) -\lim_{\tau\to\infty}\mathcal A^\gamma(M_\tau)\to0. $$
Applying the preceding estimate with $\tau_0=\tau_i$ and letting $i\to\infty$ yields
$$\lim_{i\to\infty}\sup_{\tau\ge\tau_i}\sup_{s\in[0,1]} d(M_{\tau-s},\mathscr C_+)=0.$$
Hence $\mathscr C_+$ is the unique tangent flow.

\bibliographystyle{plain}
\bibliography{references}

\end{document}